\documentclass[11pt,leqno,oneside]{amsart}

\usepackage[utf8]{inputenc} 
\usepackage[T1]{fontenc}    
\usepackage{microtype} 
\usepackage[english]{babel} 

\usepackage[a4paper,margin=1in,marginpar=1in]{geometry} 
\usepackage{csquotes}       

\usepackage[shortlabels,inline]{enumitem}

\usepackage{amsmath,amssymb}   
\usepackage{amsthm}    
\usepackage{mathtools} 
\usepackage{mathrsfs,dsfont}
\usepackage[lcgreekalpha]{stix2} 

\usepackage{xcolor}                   
\definecolor{green}{HTML}{2ECC71}
\definecolor{blue}{HTML}{3498DB}
\definecolor{red}{HTML}{E74C3C}
\definecolor{orange}{HTML}{FD6A02}
\usepackage[pdfpagelabels]{hyperref}  
\hypersetup{
  colorlinks,
  linktocpage,
  urlcolor  = blue,
  citecolor = green,
linkcolor = red}

\usepackage[capitalise,sort]{cleveref} 

\AfterEndEnvironment{proof}{\noindent\ignorespaces} 
\makeatletter
\def\@endtheorem{\endtrivlist}
\makeatother

\Crefname{paragraph}{\S}{\SS}
\Crefname{equation}{}{}
\Crefname{enumi}{}{}
\Crefname{conditioni}{Condition}{Conditions}
\Crefname{conditionalti}{Condition}{Conditions}
\newtheorem{theorem}{Theorem}[section]
\newtheorem*{theorem*}{Theorem}
\Crefname{theorem}{Theorem}{Theorems}

\Crefname{theoremintro}{Theorem}{Theorems}

\newtheorem{lemma}[theorem]{Lemma}
\Crefname{lemma}{Lemma}{Lemmas}

\Crefname{proposition}{Proposition}{Propositions}

\Crefname{corollary}{Corollary}{Corollaries}

\Crefname{conjecture}{Conjecture}{Conjectures}
\newtheorem{assumption}{Assumption}

\theoremstyle{definition}

\Crefname{example}{Example}{Examples}
\newtheorem*{example*}{Example}
\Crefname{assumption}{Assumption}{Assumptions}

\Crefname{definition}{Definition}{Definitions}

\Crefname{question}{Question}{Questions}
\theoremstyle{remark}
\newtheorem{remark}[theorem]{Remark}
\Crefname{remark}{Remark}{Remarks}

\numberwithin{equation}{section} 

\usepackage{booktabs} 

\DeclarePairedDelimiter{\paren}{\lparen}{\rparen}
\DeclarePairedDelimiter{\bracket}{\lbrack}{\rbrack}
\DeclarePairedDelimiter{\set}{\lbrace}{\rbrace}
\DeclarePairedDelimiter{\abs}{\lvert}{\rvert}

\DeclarePairedDelimiter{\norm}{\lVert}{\rVert}
\DeclarePairedDelimiterX{\psh}[2]{\langle}{\rangle}{#1, #2}
\DeclarePairedDelimiterX{\pairing}[2]{\langle}{\rangle}{#1 \vert #2}

\DeclarePairedDelimiterXPP{\Exp}[1]{\exp}{\lparen}{\rparen}{}{#1}
\DeclarePairedDelimiterXPP{\Log}[1]{\log}{\lparen}{\rparen}{}{#1}
\DeclarePairedDelimiterXPP{\Inf}[1]{\inf}{\lbrace}{\rbrace}{}{#1}
\DeclarePairedDelimiterXPP{\Sup}[1]{\sup}{\lbrace}{\rbrace}{}{#1}
\DeclarePairedDelimiterXPP{\Max}[1]{\max}{\lbrace}{\rbrace}{}{#1}
\DeclarePairedDelimiterXPP{\Min}[1]{\min}{\lbrace}{\rbrace}{}{#1}

\DeclareMathOperator{\tr}{Tr}

\DeclareMathOperator{\dom}{\mathbb{D}om}

\DeclareMathOperator{\esp}{\mathbf{E}}
\DeclareMathOperator{\prob}{\mathbf{P}}
\DeclareMathOperator{\var}{\mathbf{Var}}
\DeclareMathOperator{\law}{\mathbf{law}}

\newcommand{\given}[1][]{%
  \nonscript\:#1\vert
  \allowbreak
  \nonscript\:
\mathopen{}}

\DeclarePairedDelimiterXPP{\Prob}[1]{\prob}[]{}{#1}
\DeclarePairedDelimiterXPP{\Esp}[1]{\esp}[]{}{#1}
\DeclarePairedDelimiterXPP{\Var}[1]{\var}[]{}{#1}
\DeclarePairedDelimiterXPP{\Law}[1]{\law}[]{}{#1}

\author[J.\ ANGST]{Jürgen ANGST}
\address{IRMAR, Université de Rennes 1}
\email{jurgen.angst@univ-rennes.fr}
\author[R.\ Herry]{Ronan HERRY}
\address{IRMAR, Université de Rennes 1}
\email{ronan.herry@univ-rennes.fr}
\urladdr{https://orcid.org/0000-0001-6313-1372}
\author[D.\ Malicet]{Dominique MALICET}
\address{LAMA, Université Gustave Eiffel}
  \email{dominique.malicet@univ-eiffel.fr}
  \urladdr{https://orcid.org/0000-0003-2768-0125}
  \author[G.\ Poly]{Guillaume POLY}
\address{IRMAR, Université de Rennes 1}
  \email{guillaume.poly@univ-rennes.fr}

\title[Total variation bounds for fluctuations of linear statistics of $\beta$-ensembles]{Sharp total variation rates of convergence \\for fluctuations of linear statistics of $\beta$-ensembles}

\begin{document}

\begin{abstract}
In this article, we revisit the question of fluctuations of linear statistics of beta ensembles in the single cut and non-critical regime for general potentials $V$ under mild regularity and growth assumptions. Our main objective is to establish sharp quantitative Central Limit Theorems (CLT) for strong distances, such as the \textit{total variation} distance, which to the best of our knowledge, is new for general potentials, even qualitatively. Namely, setting $\mu_V$ the equilibrium measure, for a test function $\xi \in \mathscr{C}^{14}$, we establish the convergence in total variation of $X_n=\sum_{i=1}^n \xi(\lambda_i)-n\langle \xi,\mu_V\rangle$ to an explicit Gaussian variable at the sharp speed $1/n$. Under the same assumptions, we also establish multivariate CLTs for vectors of linear statistics in $p-$Wasserstein distances for any $p\ge 1$, with the optimal rate $1/n$, a result which already in dimension one sharpens the speed of convergence established in the recent contribution \cite{LLW} as well as the required regularity on the test functions.
A second objective of this paper, in a more qualitative direction, is to establish the so-called \textit{super-convergence} of linear statistics, that is to say the convergence of all derivatives of the densities of $X_n$ uniformly on $\mathbb{R}$, provided that $\xi\in\mathscr{C}^\infty(\mathbb{R})$ and is not too degenerated in some sense. 
~
\smallskip
\\
As far as quantitative results are concerned, our approach shares similar tools with the latter reference as we also rely on Stein's method and integration by part formalism provided by the \textit{Dyson generator} $\mathsf{L}$. Nevertheless, we adopt a different and innovative strategy which consists in proving that, up to centering, the above linear statistics $X_n$ can be written in the form $\mathsf{L} F_n+Z_n$ where $Z_n$ is some small remainder and $\Gamma[X_n,F_n]$ converges to some constant as direct consequences of the law of large numbers. We emphasize that this approach bypasses the costly requirement of the invertibility of the Markov operator which is central in \textit{Malliavin-Stein's} method and seems robust enough to be implemented for various models of Gibbs measures. On the qualitative side, the proof of the super-convergence of densities, for its part, is a consequence of the existence of negative moments for the carr\'e du champ of linear statistics, associated with integration by parts techniques.
\end{abstract}
\maketitle%
\setcounter{tocdepth}{1}
\tableofcontents%

\section{Introduction and statement of the results}

\subsection{Overview of our contributions}

Given a parameter $\beta > 0$ interpreted  as an inverse temperature, the \emph{$\beta$-ensemble}, with $n$ particles associated with a continuous potential $V \colon \mathbb{R} \to \mathbb{R}$ such that $\liminf_{x \to \pm \infty} V(x) - \log \abs{x} > 0$, is the Gibbs probability distribution on $\mathbb{R}^{n}$ associated with the \emph{energy}
\begin{equation*}
  H_{n}(\lambda) := \sum_{i < j} \log \frac{1}{\abs{\lambda_{i} - \lambda_{j}}} + n \sum_{i} V(\lambda_{i}).
\end{equation*}
Namely, the $\beta$-ensemble is the probability measure
\begin{equation}\label{eq:beta}
  \prob(\mathtt{d} \lambda) = \prob_{n, \beta}(\mathtt{d} \lambda) := \frac{1}{Z_{n,\beta}} \mathrm{e}^{-\beta H_{n}(\lambda)} \mathtt{d} \lambda,
\end{equation}
where $Z_{n,\beta} := \int \mathrm{e}^{-\beta H_{n}} \mathtt{d} \lambda$ is the \emph{partition function}, and $\mathtt{d} \lambda$ stands for the Lebesgue measure on $\mathbb{R}^{n}$.
In this paper, we establish both qualitative and quantitative Central Limit Theorems (CLT) for \emph{linear statistics} of $\beta$-ensemble, which are random variables of the form
\begin{equation*}
  X_{n} := \sum_{i=1}^{n} \xi(\lambda_{i}) - n \int \xi \mathtt{d} \mu_{V}, \qquad n \in \mathbb{N},
\end{equation*}
for some non-constant test function $\xi \colon \mathbb{R} \to \mathbb{R}$ satisfying some mild regularity assumptions, where $\mu_{V}$ is the so-called \emph{equilibrium measure}, and where the random vector $(\lambda_{1}, \dots, \lambda_{n})$ is drawn from the Gibbs measure $\prob$.
\par
\medskip
Our first main results quantifies precisely the fluctuations of such linear statistics under mild assumptions on the potential $V$ and the test function $\xi$.
Whenever $V \in \mathscr{C}^{7}(\mathbb{R})$ is semi-convex, regular, and that we are in the single-cut regime (see \cref{ass:single-cut} below), we indeed obtain a CLT for linear statistics, in both total variation and $p-$Wasserstein distances ($1 \leq p < \infty$), with the following rates:
\begin{itemize}[nosep,wide]
  \item (\cref{th:normal-approximation-optimal}) If $\xi \in \mathscr{C}^{{14}}(\mathbb{R})$, we obtain the optimal rate $\frac{1}{n}$.
  \item (\cref{th:normal-approximation}) If $\xi \in \mathscr{C}^{6}(\mathbb{R})$, we obtain the almost optimal rate $\frac{1}{n^{1-\alpha}}$ for all $\alpha > 0$.
  \item (\cref{th:normal-approximation-lipschitz}) If $\xi \in \mathscr{C}^{1,\gamma}(\mathbb{R})$ with $\gamma \in (0,1)$, we obtain an explicit rate of convergence depending on the Hölder exponent $\gamma$.
\end{itemize}
\par
\medskip
Our bounds in $p-$Wasserstein distance are furthermore obtained in the multivariate case.
In fact, we obtain precise rates of convergence, with explicit upper bounds involving $n$, $\beta$ and norms of derivatives of $\xi$. This could allow for limit theorems with varying temperature and/or with a test functions $\xi$ also depending on $n$.
\par
\medskip
Our next main result Theorem \ref{th:super-convergence} is of qualitative nature and show that the convergence in law of a linear statistic to a Gaussian automatically upgrades to a very strong form of convergence. Namely, if the potential $V$ and the test function $\xi$ are of class $\mathscr{C}^{\infty}$ and if the latter is not degenerate is some sense,  then $(X_{n})$ \emph{super-converges} to a Gaussian (see Section \ref{sec.super.state} below for definition of the super-convergence).
Informally it means that the density of $(X_{n})$ converges in the $\mathscr{C}^{\infty}$-topology to that of a Gaussian. In particular, it entails that the convergence of linear statistics holds in relative entropy.

\subsection{Motivations}

The $\beta$-ensembles, also known as \emph{$1d$-log gas}, are related to random matrix theory.
We refer to the monographs \cite{PasturShcherbina,Forrester} for details on the subject, as well as \cite{Serfaty} for a broader introduction to Coulomb gases.
Indeed, in the case $\beta \in \{1,2,4\}$ and $V(x)=x^2$, the corresponding $\beta$-ensemble describes the joint law of the spectrum of an $n \times n$ random matrix whose density is proportional to $\exp(-\frac{\beta n}{4} \tr V(M)) \mathtt{d} M$, where $\mathrm{d} M$ is the Haar measure on the sets of symmetric, hermitian, or symplectic matrices respectively.
The observation that certain $\beta$-ensembles relate to eigenvalues of Gaussian ensembles goes back at least to \cite{Dyson}.

The spectral macroscopic properties of large Gaussian matrices first observed by \cite{Wigner}, actually extend to more general $\beta$-ensembles, and it is by now well-understood, see for instance \cite[Thm.\ 11.1.2]{PasturShcherbina}, that
\begin{equation*}
  \frac{1}{n} \sum_{i=1}^{n} \xi(\lambda_{i}) \xrightarrow[n \to \infty]{a.s} \int \xi \mathtt{d} \mu_{V},
\end{equation*}
where the non-random \emph{equilibrium measure} $\mu_{V}$ is the unique minimizer of the \emph{mean-field energy}
\begin{equation*}
  \mathcal{I}_{V}(\mu) := \int V \mathtt{d} \mu - \frac{1}{2} \iint \log \abs{x-y} \mu(\mathtt{d}x) \mu(\mathtt{d} y).
\end{equation*}

In this article, our goal is to study the fluctuations associated with the above ``law of large numbers'', with different motivations that we detail below. 
\subsubsection{Sharp quantitative Central Limit Theorems}
The seminal work of \cite{Johansson} shows that, in the matrix case and under suitable assumptions on the potential $V$ and the test function $\xi$, the fluctuations of $\sum_{i=1}^{n} \xi(\lambda_{i})$ around the value $n \int \xi \mathtt{d} \mu_{V}$ are Gaussian. The quantitative aspect of this convergence has attracted a lot of attention in the last decade: $i)$ in \cite{Chatterjee}, total variation bounds are provided in the matrix case via second order Poincar\'e inequalities; $ii)$ in \cite{Bao23}, still in the matrix case, near optimal rates of convergence are derived for Kolmogorov distance; $iii)$ for general $\beta$-ensembles, polynomial rates of convergence are obtained for $2-$Wasserstein distance in \cite{LLW} using a variant of Stein's method due to Meckes.
However, in the last reference, the rates obtained are at most $O(n^{-2/3+\varepsilon})$ for general potential and of near optimal order $O(n^{-1+\varepsilon})$ only in the quadratic case. 
Our primary motivation in this work is to derive total variation estimates, which is a much stronger distance than the Wasserstein ones, but also to get optimal rate of convergence of order $O(n^{-1})$.
We also provide sharp rates of convergence in Wasserstein metric in both univariate and multivariate settings, and our findings require less regularity on the test functions than previous results on this question, see \cite{LLW,BLS}.

\subsubsection{Implementing a robust Stein's approach for \texorpdfstring{$\beta$}{β}-ensembles}
Historically, Stein's method \cite{Stein} has proved its remarkable efficientcy to obtain quantitative CLTs in strong probabilistic metrics such as total variation or relative entropy.
Despite numerous applications to various probabilistic models such as:
\begin{enumerate*}[(a)]
  \item Erdös--Rényi random graphs \cite{Ross};
  \item infinite-dimensional Gaussian fields \cite{NourdinPeccati};
  \item Poisson point processes \cite{PeccatiReitzner};
  \item free probability \cite{FathiNelson};
\end{enumerate*}
to the best of our knowledge, Stein's method to study $\beta$-ensemble has only been implemented in \cite{LLW,HG21}, through a variant due to Meckes \cite{Meckes}, which in this context consists in approaching linear statistics by infinite sums of approximate eigenvectors of the Dyson generator. This demanding task requires strong regularity assumptions on the test function $\xi$ and prevents one to both get optimal rates and to handle total variation metric. A second motivation for this works is to develop a Stein's method for $\beta$-ensembles, sufficiently robust and general  
to be possibly applied to other models in statistical physics involving Gibbs measures.

\subsubsection{Super-convergence phenomenon}
For sums of independent and identically distributed random variables, due to the convolution structure and  provided the common law has some small initial regularity, the regularity in fact improves along the classical CLT. In particular, it is possible to reinforce the classical convergence in law to the get the $\mathscr C^{\infty}$ convergence of densities, see \cite{LionsToscani}.
This particular behavior also appears in the free CLT \cite{BercoviciVoiculescu}, from which the terminology of \emph{super-convergence} proceeds.
Recently, three of the authors have revisited this regularization phenomenon on Wiener chaoses \cite{HMPSuper} and for quadratic forms \cite{HMPQuadratic}.
In the setting of $\beta$-ensembles, despite the linear nature of the statistics in consideration, no convolution structure arises due to the interaction between the particles. Establishing better-than-expected limit theorems in this dependent setting serves as the third motivation of this article.

\subsection{Statement of the main results}

Let us now precise our assumptions, fix our notations and state formally our main results.

\subsubsection{Assumption and notations}\label{sec:notations} In the rest of the paper, we will always assume that the potential $V$ satisfies the following conditions. Note that the latter are classical, in particular they coincide with the assumptions in \cite{LLW}.

\begin{assumption}\label{ass:single-cut}The potential $V$ and associated equilibrium measure $\mu_{V}$ are such that
  \begin{itemize}
    \item (Smoothness) $V \in \mathscr{C}^{7}(\mathbb{R})$.
    \item (Single-cut) The equilibrium measure $\mu_{V}$ is supported on the interval $\Sigma_V:=[-1,1]$.
    \item (Semi-convexity) $\inf_{\mathbb{R}} V'' > - \infty$.
    \item (Regularity) The measure $\mu_{V}$ has a positive density with respect to the semi-circle law.
  \end{itemize}
\end{assumption}
Recall that the measure $\mu_V$ is here the unique minimizer of the mean-field energy $ \mathcal{I}_{V}$. In particular, it satisfies the Euler--Lagrange equation, for some constant $C_{V} \in \mathbb{R}$:
\begin{equation*}
  V(x)- \int \log(|x-y|) \mu_V(\mathtt{d} y) = C_V, \qquad \forall x \in \Sigma_{V}.
\end{equation*}
Under the above assumptions, the equilibrium measure $\mu_V$ admits a density with respect to the semi-circular law $\mu_{sc}(\mathtt{d}x):=\frac{2}{\pi} \sqrt{1-x^2}\mathds{1}_{[-1,1]}(x)\mathtt{d}x$, namely  setting $\rho(\mathtt{d}x):= \frac{\mathtt{d}x}{\pi \sqrt{1-x^2}}$
\[
\mu_V(\mathtt{d}x) = S(x)\mu_{sc}(\mathtt{d}x), \quad \text{where} \quad S(x):=\frac{1}{2} \int_{-1}^1 \frac{V'(x)-V'(y)}{x-y} \rho(\mathtt{d}y).
\]
As already mentioned, our main objective is to provide (multivariate) qualitative and quantitative CLTs for linear statistics of $\beta$-ensembles.
We consider test functions $\xi_{1}, \dots, \xi_{d} \colon \mathbb{R} \to \mathbb{R}$ and we define the random vector $X = (X_{1}, \dots, X_{d})$ by
\begin{equation*}
  X_{k} := \sum_{i=1}^{n} \xi_{k}(\lambda_{i}) - n \int \xi_{k} \mathtt{d} \mu_{V}, \quad 1\leq k \leq d.
\end{equation*}
For such a vector $X$, as $n$ goes to infinity, following Equations $(1.10)$ and $(1.11)$ in \cite{LLW}, the candidate limiting mean $m=(m_i)_{1\leq i \leq d}$ and covariance matrix $C=(c_{i,j})_{1\leq i,j,d}$ are given by
\begin{equation*}
m_i:=\left(  \frac{1}{2}-\frac{1}{\beta}\right)\left[\frac{\xi_i(-1)+\xi_i(1)}{2}  -\int_{\Sigma_V} \xi_i(x)\rho(\mathtt{d}x)-\frac{1}{2} \int_{\Sigma_V^2} \frac{S'(x)}{S(x)} \frac{\xi_i(x)-\xi_i(y)}{x-y} \rho(\mathtt{d}y) \mu_{sc}(\mathtt{d}x)\right],
\end{equation*}
\begin{equation*}
  c_{ij} := \frac{1}{2\beta} \int \frac{\xi_{i}(x) - \xi_{i}(y)}{x - y} \frac{\xi_{j}(x) - \xi_{j}(y)}{x - y} (1 - xy) \rho(\mathtt{d}x)\rho(\mathtt{d}y).
\end{equation*}
The above matrix $C$ is well-defined as soon as the $\xi_{i}$'s are $1/2$-Hölder continuous.
We shall see in Section \ref{sec.invert} below that $C$ is indeed a covariance matrix as soon as 
\begin{equation}\label{eq:free}
  1, \xi_{1}, \dots, \xi_{d} \ \text{are linearly independent}.
\end{equation}

For an open set $U \subset \mathbb{R}$, we write $\mathscr{C}^{r}(U)$ for the space of $r$ times continuously differentiable functions on $U$ with bounded derivatives.
We endow it with the Banach norm
\begin{equation*}
\norm{\xi}_{\mathscr{C}^{r}(U)} := \max_{r' \leq r} \sup_{x \in U} \abs{\xi^{(r')}(x)}.
\end{equation*}
When $\xi = (\xi_{1}, \dots, \xi_{d})$, we also write
\begin{equation*}
  \norm{\xi}_{\mathscr{C}^{r}(U)} = \sum_{k=1}^{d} \norm{\xi_{k}}_{\mathscr{C}^{r}(U)}.
\end{equation*}
We write indifferently $\abs{\cdot}$ for the absolute value of a real number, the Euclidean norm of a vector, or the Euclidean norm of a square matrix, also known as its Hilbert--Schmidt norm.
For a square symmetric matrix $A$, we write $\norm{A}_{op}$ for its \emph{operator norm}, that is its largest singular value.
Given two random variables $X$ and $Y$, the \emph{$p-$Wasserstein distance} is
\begin{equation*}
  \mathbf{W}_{p}(X, Y) := \Inf*{ \Esp*{ \abs{X' - Y'}^{p} } : X' \overset{\law}{=} X, Y' \overset{\law}{=} Y  }.
\end{equation*}
The Wasserstein distances only depends on $X$ and $Y$ through their respective laws but we favor a probabilistic notation more suited for the approximation theorems we establish.
The $p-$Wasserstein distance induces a topology corresponding to convergence in law together with convergence of the $p$-th moment \cite[Thm.\ 6.9]{Villani}.
We also work with the total variation distance of real-valued random variables
\begin{equation*}
  \begin{split}
    \mathbf{TV}(X,Y) & := \Sup*{ \Prob{X \in B} - \Prob{Y \in B} : B \ \text{Borel} }. 
  \end{split}
\end{equation*}

Let us finally define
\begin{align*}
  & A_{\beta} := \paren*{ \frac{1}{\beta}  \norm{C^{1/2}}_{op} \abs{C^{-1}} \norm{N}_{L^{p}} + \abs*{ \frac{1}{2} - \frac{1}{\beta} } \norm{C^{1/2}}_{op}};
\\& a_{\beta} := \paren*{ \frac{1}{\beta} \frac{1}{\sigma^{2}} + \abs*{ \frac{1}{2} - \frac{1}{\beta} } \frac{\sqrt{\pi}}{2} }.
\end{align*}
\subsubsection{Quantitative CLT in total variation and Wasserstein distances} 
We can now state our main results concerning the quantitative behavior of the fluctuations of linear statistics of $\beta$-ensembles. The following bounds are optimal or nearly optimal depending on the regularity of the test functions. 
 
\begin{theorem}\label{th:normal-approximation}
Let us fix $1 \leq p < \infty$ and $\alpha > 0$. There exists $K_{p,\alpha} > 0$ such that
for all test functions $\xi_{1}, \dots, \xi_{d} \in \mathscr{C}^{{6}}(\mathbb{R})$ satisfying \cref{eq:free}, we have, with $N$ a standard Gaussian in $\mathbb R^d$
  \begin{equation*}
    \mathbf{W}_{p}(X, C^{1/2}N + m) \leq K_{p,\alpha} A_{\beta} \frac{\norm{\xi}_{\mathscr{C}^{6}(\mathbb{R})} + \norm{\xi}_{\mathscr{C}^{2}(\mathbb{R})}^{2}}{n^{1-\alpha}} .
  \end{equation*}
  In the univariate case $d=1$, with $\sigma:=\sqrt{c_{11}}$ and $m=m_1$, we also have
  \begin{equation*}
    \mathbf{TV}(X, \sigma N + m) \leq K_{1,\alpha} a_{\beta} \frac{\norm{\xi}_{\mathscr{C}^{6}(\mathbb{R})} + \norm{\xi}_{\mathscr{C}^{2}(\mathbb{R})}^{2}}{n^{1-\alpha}}.
  \end{equation*}
\end{theorem}

Assuming more regularity on the test functions, we obtain an optimal rate of convergence. 
\begin{theorem}\label{th:normal-approximation-optimal}
For $1 \leq p < \infty$ and $\alpha > 0$, there exists $K_{p,\alpha} > 0$ such that
for all test functions  $\xi_{1}, \dots, \xi_{d} \in \mathscr{C}^{{14}}(\mathbb{R})$ satisfying \cref{eq:free}, we have, with $N$ a standard Gaussian in $\mathbb R^d$
  \begin{equation*}
    \mathbf{W}_{p}(X, C^{1/2}N + m) \leq K_{p} A_{\beta} \frac{\norm{\xi}_{\mathscr{C}^{14}(\mathbb{R})} + \norm{\xi}_{\mathscr{C}^{7}(\mathbb{R})}^{2}}{n} .
  \end{equation*}
  In the univariate case $d=1$, with $\sigma:=\sqrt{c_{11}}$ and $m=m_1$, we also have
  \begin{equation*}
    \mathbf{TV}(X, \sigma N + m) \leq K_{1} a_{\beta} \frac{\norm{\xi}_{\mathscr{C}^{14}(\mathbb{R})} + \norm{\xi}_{\mathscr{C}^{7}(\mathbb{R})}^{2}}{n}.
  \end{equation*}
\end{theorem}

We also prove a theorem for functions of lower regularity.
For $\gamma \in (0,1)$, we write $\xi \in \mathscr{C}^{1,\gamma}(\mathbb{R})$, provided $\xi \in \mathscr{C}^{1}(\mathbb{R})$ and it satisfies
\begin{equation*}
  \norm{\xi}_{\mathscr{C}^{1,\gamma}(\mathbb{R})} := \norm{\xi}_{\mathscr{C}^{1}(\mathbb{R})} + \sup_{x \ne y} \frac{\abs{\xi'(x) - \xi'(y)}}{\abs{x - y}^{\gamma}} < \infty.
\end{equation*}

\begin{theorem}\label{th:normal-approximation-lipschitz}
Let $\xi \in \mathscr{C}^{1,\gamma}(\mathbb{R})$ for some $\gamma \in (0,1)$. Setting $\sigma:=\sqrt{c_{11}}$ and $m=m_1$, and for any $a<\frac{\gamma}{6+\gamma}$, we have
\begin{equation*}
  \mathbf{TV}(X, \sigma N + m) \leq K_{a} a_{\beta} \frac{\norm{\xi}_{\mathscr{C}^{1,\gamma}(\mathbb{R})}}{n^{a}}.
  \end{equation*}
\end{theorem}

\subsubsection{Super-convergence}\label{sec.super.state}
We now turn to our qualitative results and the reinforcement of the mode of convergence is the smooth case. 
We say that a sequence $(X_{n})$ \emph{super-converges} to a non-degenerate Gaussian $N\sim \mathcal N(m,\sigma^2)$ with density $\varphi_{m,\sigma^2}$ provided that,
for all $r \in \mathbb{N}$, there exists $n_0 \in \mathbb{N}$ such that for $n \geq n_0$, the law of $X_{n}$ admits a density $\varphi_{n} \in \mathscr{C}^{r}(\mathbb{R})$, and

\begin{equation*}
\limsup_{n \to \infty}  \norm{ \varphi_{n} - \varphi_{m,\sigma^2}}_{\mathscr{C}^{r}(\mathbb{R})} = 0.
\end{equation*}

Regarding linear statistics, we obtain that convergence in law can easily be upgraded to super-convergence.
For $\alpha > 0$, we say that $\xi \in \mathscr{C}^{1}(\mathbb{R})$ is \emph{$\alpha$-regular} provided
that $$\mathsf{Leb}(x \in \mathbb{R}^{d} : |\xi'(x)| \le \epsilon) \lesssim \epsilon^{\alpha}.$$

It is in particular true for any polynomial function $\xi$.

\begin{theorem}\label{th:super-convergence}
Let $\xi \in \mathscr{C}^{\infty}(\mathbb{R})$ be $\alpha$-regular such that $(X_{n})$ converges in law to a Gaussian variable $N \sim \mathcal N(m, \sigma^2)$. Then, $(X_{n})$ super-converges to $N$.
\end{theorem}
%
%

\newpage
\subsection{Comparison with existing results}

\subsubsection{Gaussian fluctuations for linear statistics}
The mathematical study of fluctuations of linear statistics around their equilibrium starts with the seminal work \cite{Johansson}.
His method consists in writing the Laplace transform of a linear statistics as the ratio of a partition function associated with a perturbed potential by that associated with the initial potential.
Interestingly, the perturbation involves the so-called master operator, noted $\Theta_{V}$ in the sequel, see section \ref{master-operator-def} for precise definition. This operators plays also a central role in our approach.
The method developed in \cite{Johansson} has been continued in \cite{KriecherbauerShcherbina,Shcherbina,BorotGuionnet,BorotGuionnetMulti} to obtain qualitative CLT with increasing levels of generality, allowing, for instance, results in the multi-cut regime.
All the above results require the potential $V$ to be analytic.
Using the same approach on the Laplace transform, \cite{BLS} proves a CLT when $V \in \mathscr{C}^{5}(\mathbb{R})$ and $\xi \in \mathscr{C}^{3}_{c}(\mathbb{R})$.
To the best of our knowledge, these are the best results available regarding the regularity of $V$ and $\xi$.
In contrast, we obtain a CLT for $V \in \mathscr{C}^{7}(\mathbb{R})$ but allowing to lower the regularity to $\xi \in \mathscr{C}^{1,\gamma}(\mathbb{R})$.
Our results are also limited to the single-cut case.
We stress that in the multi-cut regime, non-Gaussian fluctuations are known, and additional compatibility conditions on $\xi$ are required to ensure a CLT.
We also mention \cite{BekermanLodhia} that derives a CLT for linear statistics at all meso-scales, that is for linear statistics where the particles are rescaled by a factor $n^{-\alpha}$ for $\alpha \in (0,1)$.
Even if our main result allows to recover CLT at some meso-scale, we cannot reach their full range with our method.

\subsubsection{Stein's method and \texorpdfstring{$\beta$}{β}-ensemble}
To the best of our knowledge, the reference \cite{LLW} is the first to implement Stein's method in the setting of \textit{general} $\beta$-ensembles and to provide quantitative bounds in $2-$Wasserstein distance which are near optimal for $V(x)=x^2$ but sub-optimal for more general potentials. Their method works for linear statistics that are close to be eigenvalues of the Dyson generator, and technically completing their requires, among others, the \textit{tour de force} of diagonalizing the so-called master operator. Relying on the concept of exchangeable pairs which plays an important role in Stein's method (see e.g. \cite{Meckes}), quantitative CLTs for linear statistics of Haar distributed random variables on compact classical groups are established in \cite{MR3245992,MR2861678} and later on, in the context of \textit{circular} $\beta$-ensembles, in \cite{MR3485367} where exchangeable pairs are built through the $n$-dimensional circular Dyson Brownian motion. Finally, the contribution \cite{HG21} which builds upon the techniques given in \cite{LLW} to provide quantitative (near optimal) CLTs in the high temperature regime.

\par
\medskip

Let us also mention the seminal article \cite{Chatterjee}.
There, the author introduced a variation around Stein's method, based on \emph{second-order Poincaré inequality} in order to study fluctuations of eigenvalues of matrices with random coefficients, possibly not identically distributed, whose distributions admit suitable densities.
It is somewhat complicated to compare the results there with ours, simply because for general potentials $V$ we cannot a priori interpret $\beta$-ensembles as spectrum of random matrices with independent coefficients. Besides, \cite[Thm. 4.2]{Chatterjee} obtains convergence to a Gaussian of random variables of the form $\tr A_{n}^{p_{n}}$ where $(A_{n})$ are some random matrices as above, and $p_{n} = o(\log n)$. Note that, in comparison, in our context, it should be possible to get CLTs for $\xi_n(x)=x^{p_n}$ when $p_n$ grows polynomially although the models are pretty different and only coincide for $V(x)=x^2$.

\par
\medskip

We would like to point out that despite the aforementioned articles and ours being inspired by the philosophy of Stein's method, they strongly differ in the way Stein's method is implemented. Indeed, \cite{LLW} relies on ideas coined in \cite{Meckes} using exchangeable pairs while our builds on a novel refinement of the well established \textit{Malliavin Stein's} approach, which allows to handle random variables which \textbf{are not} in the image of the underlying Markov generator. Indeed, whenever $X$ belongs to the image of $\mathsf{L}$ one can directly provide a \textit{Stein's kernel} by setting $\tau(x)=\mathbb{E}\left[\Gamma\left[X,-L^{-1}X\right]\right]$ which fulfils the equation $\mathbb{E}\left[\phi'(X)\tau(X)-X\phi(X)\right]=0$. The great advantage of this point of view is to control the total variation distance (among others) by $\mathbb{E}\left[|1-\tau(X)|\right]$ but it requires the costly assumption that $\mathsf{L}$ is invertible which is not true in general. To overcome this obstacle, we notice that it is enough to show that $X$ is near $\text{Im}(Z)$ in the following sense: $X=LF+Z$ for $Z$ small in probability. This simple remark, which is new to the best of our knowledge, increases considerably the applicability of the above Malliavin-Stein's approach and still enables to provide total variation bounds but is a priori not sufficient to build a Stein's kernel. While this being noticed, our strategy then consists in proving that any linear statistics is close to the range of $\mathsf{L}$, a step which proceeds from the invertibility of the master operator. As a result, our findings thus refine those of \cite{LLW} by providing optimal rates of convergence for stronger metrics such as the total variation distance and we believe that our approach could also be successfully implemented in the framework of circular $\beta$-ensembles and Haar distributed matrices on compact groups.

\subsubsection{Super-convergence}
To the best of our knowledge, in the context of random matrices or $\beta$-ensembles, the question of establishing convergence in metrics stronger than total variation or Wasserstein distances has not been considered yet. We deploy here ideas that are classical in the framework of Malliavin calculus and stochastic analysis and which merely consists in establishing negative moments for the \textit{carré-du-champ} operator applied to the considered linear statistics.  Then, relying on integration by parts techniques, we are able to prove strong forms of convergence for the densities of the underlying sequences of random variables.

\subsection{Outline of the proofs an plan of the paper}

Let us give here more details on our strategy of proof and on the plan of the paper. 

\subsubsection{The generator of the Dyson Brownian motion}
The overall strategy behind our quantitative and qualitative estimates leverage the characterization of the $\beta$-ensemble as the unique invariant distribution of the Dyson Brownian motion.
Namely, consider the \emph{generator of the Dyson Brownian motion}
\begin{equation}\label{eq:generator}
  \mathsf{L} = \mathsf{L}_{n, \beta, V}=:= \Delta - \beta \nabla H \cdot \nabla = \sum_{i=1}^{n} \partial_{\lambda_{i}}^{2} - \beta n \sum_{i=1}^{n} V'(\lambda_{i}) \partial_{\lambda_{i}} + \frac{\beta}{2n} \sum_{i \ne j} \frac{\partial_{\lambda_{i}} - \partial_{\lambda_{j}}}{\lambda_{i} - \lambda_{j}}.
\end{equation}
The operator $\mathsf{L}$ is the diffusive Markov generator canonically associated with the $\beta$-ensemble. We have indeed the following characterization of the $\beta$-ensemble Gibbs measure $\prob$:
\begin{equation*}
\tilde{\prob} = \prob \Leftrightarrow  \paren*{ \tilde{\esp} \mathsf{L} F = 0, \ \forall F \in \mathscr{C}^{\infty}_{c}(\mathbb{R}^n) } \Leftrightarrow \paren*{ \tilde{\esp} \bracket*{ F \mathsf{L} G} = \tilde{\esp}\bracket{G \mathsf{L} F}, \ \forall F,G \in \mathscr{C}^{\infty}_{c}(\mathbb{R}^n) }.
\end{equation*}
The differential structure induced by $\mathsf{L}$ on $\mathbb{R}^{n}$, technically called a \emph{Dirichlet structure}, comes with a \emph{carré du champ} operator
\begin{equation*}
\Gamma[F,G] := \nabla F \cdot \nabla G, \qquad F,\,G \in \mathscr{C}^{1}(\mathbb{R}^n).
\end{equation*}
We have the following integration by parts formula
\begin{equation}\label{eq:ipp}
  \Esp*{ \Gamma[F,G] } = -\Esp*{ F \mathsf{L} G}, \qquad F,\, G \in \mathscr{C}^{\infty}_{c}(\mathbb{R}^n).
\end{equation}

\subsubsection{\texorpdfstring{$\Gamma$}{Γ}-Stein's method for $\beta$-ensemble}
At an informal level, the data of a diffusive Markov generator generally combines well with Stein's method to provide quantitative bounds for normal approximation.
Provided $\ker \mathsf{L}$ is limited to constant functions, $\mathsf{L}$ is invertible on mean-zero functions, and the celebrated \emph{$\Gamma$-Stein} approach \cite{NourdinPeccati,PeccatiReitzner,ACP,LNP}, or \emph{Malliavin--Stein} approach in the setting of Gaussian fields or Poisson point processes, yields that the variance of $\Gamma[X, - \mathsf{L}^{-1}X]$ controls the Gaussian fluctuations of $X$, in total variation or Wasserstein distance.
However, in the case $\beta$-ensembles, the operator $\mathsf{L}$ is in general \emph{not invertible}.

To overcome this difficulty, we amend the classical $\Gamma$-Stein approach.
Intuitively, whenever $X = \mathsf{L}F$ for some $F$, $\mathsf{L}^{-1}X$ makes sense despite the non-invertibility of $\mathsf{L}$.
In this case, it is natural to expect that $\Gamma[X, -\mathsf{L}^{-1}X] = \Gamma[X, -F]$ controls the Gaussian fluctuations of $X$.
In the next Section \ref{sec.stein} and more precisely in \cref{th:stein-approximation} below, we formalize this intuition by proving quantitative bounds for the normal approximation, in Wasserstein distance and total variation, of random variables of the form $X = \mathsf{L} F + Z$.
Provided, $Z$ is small and that $\Gamma[X,-F]$ is close to a constant $\sigma^{2}$, then $X$ is close to a Gaussian with variance $\sigma^{2}$.
This theorem should not come as a surprise to Stein's method aficionados and its proof is rather straightforward.
However, this observation has, to the best of our knowledge, never been remarked and this ``almost invertibility'' decomposition is precisely what allows us to provide an efficient proof.
We stress that our bounds hold for generic abstract diffusive Markov generators and we believe it could prove useful in other statistical physics models where the Gibbs measure has an explicit density.

\subsubsection{The master equation} 
The next Section \ref{s:master-equation} of the paper consists in checking that the Dyson generator is indeed in line with the global strategy developed in Section \ref{sec.stein}. 
In order to apply our abstract bound, we need to show that our linear statistics $X$ is of the form $\mathsf{L} F + Z$. With the notations of Section \ref{s:master-equation}, using mostly algebraic manipulations together with the minimality of $\mu_{V}$ with respect to the mean-field energy $\mathcal{I}_{V}$, we show, in \cref{th:master}, that whenever $F := \sum_{i=1}^{n} f(\lambda_{i})$ is a linear statistic, setting $m_f:=\paren*{ \frac{1}{2} - \frac{1}{\beta} } \pairing{f''}{\mu_{V}}$ which is such that the term in parenthesis below is asymptotically centered, we have
\begin{equation*}
  \frac{1}{n \beta} \mathsf{L} F = \left( \sum_{i=1}^{n} (\Theta_{V}f')(\lambda_{i}) - n \int (\Theta_{V} f') \mathtt{d} \mu_{V} -m_f \right)+ \frac{Z}{\beta},
\end{equation*}
where $Z$ is a quadratic (which should be thought as a remainder) term and $\Theta_{V}$ is the so-called  \emph{master operator}.
In particular, under \cref{ass:single-cut}, $\Theta_{V}$ is invertible in a neighborhood of $[-1,1]$.
Namely, there exists an open neighborhood $U$ of $[-1,1]$, such that given $\xi \in \mathscr{C}^{6}(\mathbb{R})$, we can find $\psi \in \mathscr{C}^{5}_{c}(\mathbb{R})$ and $c_{\xi} \in \mathbb{R}$ satisfying
\begin{equation*}
  (\Theta_{V}\psi)(x) = \xi(x) + c_{\xi}, \qquad x \in U.
\end{equation*}

Under \cref{ass:single-cut}, by \cite{BEY} the $(\lambda_{i})$ enjoy a strong rigidity property, precisely recalled in \cref{th:rigidity}.
Informally, the probability that the $\lambda_{i}$'s deviate from the grid given by the $i/n$-th quantiles of $\mu_{V}$ is overwhelmingly small. Note that this rigidity phenomenon is also used in \cite{LLW}.
Thus up to paying a loss factor, that is negligible compared to our bound, we can assume that all the $\lambda_{i}$'s are localized in the neighborhood $U$.
By the master equation, choosing $f$ any primitive of $\psi$, we are thus left with establishing the quantitative CLT for the random variable
\begin{equation*}
  X = \frac{1}{n\beta} \mathsf{L} F -\frac{Z}{\beta},
\end{equation*}
which falls precisely under the scope of our Stein's bound.

\subsubsection{Controlling the carré du champ and the remainder} In Section \ref{sec.proofs}, we concretely implement our strategy and complete the proofs of our main results. We first establish the quantitative statements, namely \cref{th:normal-approximation}, \cref{th:normal-approximation-optimal} and \cref{th:normal-approximation-lipschitz}.
In view of our strategy, we show that $(n\beta)^{-1} \Gamma[X,-F]$ is close to a constant and that $(\beta)^{-1}Z$ is small.
We control the term involving the carré du champ by relying on a simple yet remarkable property of the carré du champ: it preserves linear statistics.
Precisely, we have that
\begin{equation*}
  \Gamma\bracket*{\sum \varphi(\lambda_{i}), \sum \chi(\lambda_{i})} = \sum \varphi'(\lambda_{i}) \chi'(\lambda_{i}).
\end{equation*}
It follows that
\begin{equation*}
  \frac{1}{n} \Gamma[X, F] = \frac{1}{n} \sum \xi'(\lambda_{i}) f'(\lambda_{i}).
\end{equation*}
By the convergence to equilibrium, the above quantity converges to $\int \xi' f' \mathtt{d} \mu_{V}$.
Using the minimality property of $\mu_{V}$, it can be shown that this last integral coincides with limiting covariance given above.
To derive quantitative bounds from there, we rely on a quantitative convergence to equilibrium which also follows from the rigidity estimate from \cite{BEY}.

Handling the quadratic remainder $Z$ is more tedious.
We use Fourier inversion to decompose $Z$ into a product of linear statistics.
By the aforementioned quantitative law, one of term goes to $0$ at a given rate while the other stays bounded.
This allows us to conclude for the proof \cref{th:normal-approximation}.
From there, \cref{th:normal-approximation-optimal} follows by a bootstrap argument: we essentially redo the same proof but instead of using the sub-optimal quantitative law of large numbers provided by \cite{BEY}, we use the CLT we just established that gives us a law of large numbers at speed $1/n$.
To conclude for \cref{th:normal-approximation-lipschitz}, we approximate $\xi \in \mathscr{C}^{1,\gamma}(\mathbb{R})$ by a sequence of smooth functional $\xi_{\varepsilon}$ while choosing $\varepsilon \simeq n^{a}$, for a well-chosen $a$.

\subsubsection{Super-convergence}The proof of our qualitative result Theorem \ref{th:super-convergence} is finally given in Section \ref{sec.super1}.
The derivation of the super-convergence relies on a lemma, classical in Dirichlet forms / Malliavin calculus theory, stating that negative moments of $\Gamma[X,X]$ control the Sobolev norms of the density of $X$.
We rely again on the fact that $\Gamma$ preserves linear statistics.
In this precise case,
\begin{equation*}
  \Gamma{[X,X]} = \sum (\xi'(\lambda_{i}))^{2}.
\end{equation*}
We obtain negative moments through a direct control on $\Prob*{ \sum (\xi'(\lambda_{i}))^{2} \leq \varepsilon }$.

\section{Stein's method for Markov diffusive operators}\label{sec.stein}

In order to carry out our program, we need to establish quantitative bounds in total variation for random variables which are close in some sense to the range of a given Markov diffusive operator $\mathsf{L}$. To the best of our knowledge, the following estimates seem to be new in the well-studied area of Malliavin--Stein's method and are of independent interest. We stress that all the results of this section are valid for any diffusive Markov operator $\mathsf{L}$ associated with $\Gamma$ the so-called carré du champ and $\prob$ the invariant measure for $\mathsf{L}$. We refer to \cite{BGL} for definitions in this abstract setting. Indeed, the only properties used in our proofs are the chain rule and the integration by parts which hold in full generality. The cornerstone of our method relies on the following Theorem which will proved in Section \ref{preuve-thm-stein} based on the content of Sections \ref{remainder-stein} and \ref{computation-stein}.

\medskip

In case of vector-valued random variables $F$ and $G$, we extend our definition to the \emph{matrix-valued carré du champ}
\begin{equation*}
  \Gamma[F,G]_{ij} := \Gamma[F_{i}, G_{j}], \qquad i,j = 1, \dots, d.
\end{equation*}

\begin{theorem}\label{th:stein-approximation}
  Take $F_{1}, \dots, F_{d} \in \dom \mathsf{L}$, $Z_{1}, \dots, Z_{d} \in L^{p}$, and $X := (\mathsf{L} F_{1} + Z_{1}, \dots, \mathsf{L} F_{d} + Z_{d})$.
  Let $\Sigma$ be a positive symmetric matrix, $C=\Sigma^2$ and $N$ be standard Gaussian in $\mathbb R^d$. Then, we have
\begin{equation}\label{eq:stein-bound-wasserstein}
  \begin{split}
    \mathbf{W}_{p}(X, \Sigma N) & \leq \norm{\Sigma}_{op} \norm{N}_{L^{p}} \norm{ \mathrm{id} - \Sigma^{-1} \Gamma[X, -F] \Sigma^{-1} }_{L^{p}} + \norm{\Sigma}_{op} \norm{Z}_{L^{p}}
                              \\& \leq \norm{\Sigma}_{op} \norm{C^{-1}} \norm{N}_{L^{p}} \norm{C - \Gamma[X,-F]}_{L^{p}} + \norm{\Sigma}_{op} \norm{Z}_{L^{p}}.
  \end{split}
\end{equation}
  Moreover, in the univariate case $d=1$, setting $\sigma=\Sigma_{11}$, we obtain
\begin{equation}\label{eq:stein-bound-tv}
  \mathbf{TV}(X, \sigma N) = \mathbf{TV}(\sigma^{-1}X, N) \leq \frac{2}{\sigma^{2}} \norm{ \sigma - \Gamma[X,-F] }_{L^{1}} + \frac{\sqrt{\pi}}{2} \norm{Z}_{L^{1}}.
\end{equation}
\end{theorem}

\subsection{Reminders on Stein kernels}\label{remainder-stein}
Let us first recall important results regarding Stein kernels.
Here we work on a general probability space $(\Omega, \mathfrak{W}, \prob)$.
Given a multivariate random variable $X = (X_{1}, \dots, X_{d})$, we say that a matrix-valued random variable $\tau$, measurable with respect to $X$, is a \emph{Stein kernel} for $X$ provided
\begin{equation*}
  \Esp*{X \cdot \nabla \varphi(X)} = \Esp*{ \tau \cdot \nabla^{2} \varphi(X) }, \qquad \varphi \in \mathscr{C}^{\infty}(\mathbb{R}^{d}).
\end{equation*}
Heuristically, Stein kernels are relevant for normal approximation, since $X$ is a standard multivariate normal variable if and only if it admits $\tau = \mathrm{id}$ as a Stein kernel.
A key observation at the heart of Stein's method are the following quantitative normal approximation inequalities.

\begin{lemma}[Stein kernel bounds \cite{LNP,Fathi,nourdin2009stein} ]\label{th:stein-bound}
  Let $X = (X_{1}, \dots, X_{d})$, let $N$ be a standard normal variable on $\mathbb{R}^{d}$, and $\tau$ be Stein kernel for $X$.
  \begin{enumerate}[(i), wide]
    \item Let $p \in [1, \infty)$.
      If $X \in L^{p}(\Omega)$, then
      \begin{equation*}
        \mathbf{W}_{p}(X,N)^{p} \leq \Esp*{ \abs{N}^{p} } \Esp*{ \norm{ \tau - \mathrm{id}}^{p} }.
      \end{equation*}
    \item  In the univariate setting $d=1$, we have
\begin{equation*}
  \mathbf{TV}(X, N) \leq 2 \Esp*{ \abs{\tau -1} }.
\end{equation*}
\end{enumerate}
\end{lemma}



\subsection{Computations of Stein kernels on \texorpdfstring{$\operatorname{Im} \mathsf{L}$}{Im L}}\label{computation-stein}

The celebrated Malliavin--Stein method \cite{NourdinPeccati} puts forward a strikingly efficient way to compute a Stein kernel for sufficiently smooth functionals of an infinite-dimensional Gaussian field.
\begin{lemma}\label{th:stein-kernel-L}
  Take $F_{1}, \dots, F_{d} \in \dom \mathsf{L}$.
  Let $X := (\mathsf{L} F_{1}, \dots, \mathsf{L} F_{d})$, and define the matrix-valued random variable $\tau = (\tau_{ij})$, where
  \begin{equation*}
    \tau_{ij} := \Esp*{ \Gamma[X_{i}, -F_{j}] \given X }, \qquad i,j = 1,\dots,d.
  \end{equation*}
  Then, $\tau$ is a Stein kernel for $X$.
\end{lemma}
\begin{proof}
  Take $\varphi \in \mathscr{C}^{\infty}(\mathbb{R}^{d})$.
  Then, using the integration by parts, and then the chain rule, we get
  \begin{equation*}
    \Esp*{ X \cdot \nabla \varphi(X) } = \sum_{i=1}^{l} \Esp*{ \mathsf{L} F_{i} \partial_{i} \varphi(X) } = - \sum_{i,j} \Esp*{ \partial_{ij} \varphi(X) \Gamma[F_{i}, X_{j}] }.
  \end{equation*}
  This concludes the proof in view of the definition of a Stein kernel.
\end{proof}

Combining \cref{th:stein-bound,th:stein-kernel-L}, we obtain the following $\Gamma$-Stein bound.

\begin{theorem}
  Take $F_{1}, \dots, F_{d} \in \dom \mathsf{L}$, and $X := (\mathsf{L} F_{1}, \dots, \mathsf{L} F_{d})$.
  Then
  \begin{equation*}
    \mathbf{W}_{p}(X, N)^{p} \leq \Esp*{ \abs{N}^{p} } \Esp*{ \norm{\operatorname{id} - \Gamma[X, -F]}^{p} }.
  \end{equation*}
  Moreover, in the univariate case $d=1$, we find
\begin{equation*}
  \mathbf{TV}(X,N)\le 2 \Esp*{ \abs{1-\Gamma[X,-F]} }.
\end{equation*}
\end{theorem}

\begin{remark} The classical approach on the Wiener space relies on the invertibility of the generator of the Ornstein--Uhlenbeck process.
  In this case, it is known that
  \begin{equation*}
    \tau_{ij} := \Esp*{ \Gamma[X_{i}, -\mathsf{L}^{-1}X_{j}] \given X }, \qquad i,j = 1, \dots, d,
  \end{equation*}
  is a Stein kernel for $X \in L^{2}(\Omega)$ with $\esp X = 0$.
  This strategy has also been applied in the setting of other diffusive Markov generators with discrete spectrum, see for instance \cite{ACP,LNP}.
  This actually could be further generalized to any diffusive Markov generators that is invertible away from constants.
  Note that, the generator of the Dyson Brownian motion $\mathsf{L}$ is in general not invertible, in particular it is a priori not true that a given linear statistics belongs to the range of $\mathsf{L}$. Therefore, we need to extend the above techniques for random variables not belonging to the range of $\mathsf{L}$ which is precisely the object of the next section.
\end{remark}

\subsection{Normal approximation away from \texorpdfstring{$\operatorname{Im} \mathsf{L}$}{Im L}}\label{preuve-thm-stein}

Given $X$, finding $F$ such that $\mathsf{L} F = X$ is a demanding task.
Indeed as already mentioned, $\mathsf{L}$ may be not invertible. Nevertheless, as we now demonstrate, the bounds obtained by the Stein's method can be amended in order to cover the case where $X$ is close to the range of $\mathsf{L}$ that is to say of the form $X = \mathsf{L} F + Z$ with small perturbation $Z$.

\begin{proof}[Proof of {\cref{th:stein-approximation}}]
  We first prove the claim in the case $\Sigma = \mathrm{id}$, and then explain how to handle the case of general non-degenerate covariance.
  For the upper bound \cref{eq:stein-bound-wasserstein}, by the triangle inequality, we have that
  \begin{equation*}
    \mathbf{W}_{p}(X, N) \leq \mathbf{W}_{p}(\mathsf{L}F, N) + \mathbf{W}_{p}(\mathsf{L}F, \mathsf{L}F + Z).
  \end{equation*}
  We apply \cref{th:stein-kernel-L} on the first term on the right-hand side.
  For the second term by choosing the coupling $(\mathsf{L}F, \mathsf{L}F+Z)$ in the minimization problem defining $\mathbf{W}_{p}$, we see that
  \begin{equation*}
    \mathbf{W}_{p}(\mathsf{L}F, \mathsf{L}F+Z)^{p} \leq \Esp*{ \norm{\mathsf{L}F - \mathsf{L}F +Z}^{p} } = \Esp*{ \norm{Z}^{p}}.
  \end{equation*}
  This concludes the proof of \cref{eq:stein-bound-wasserstein}.

  We now prove \cref{eq:stein-bound-tv}.
Since, by the chain rule an integration by parts, we have
\begin{equation*}
  \Esp*{ \varphi'(X)\Gamma[X,-F] } = \Esp*{ \varphi(X) \mathsf{L} F }=\Esp*{ \varphi(X) X }-\Esp*{ \varphi(X) Z },
\end{equation*}
we find that
\begin{equation*}
\Esp*{ \varphi'(X)-X\varphi(X) } = \Esp*{ \varphi'(X)(1-\Gamma[X,-F]) }
 - \Esp*{ \varphi(X) Z }.
\end{equation*}
  Recalling that, by \cite[Thm.\ 3.3.1]{NourdinPeccati},
\begin{equation*}
  \mathbf{TV}(X,N) \leq \Sup*{ \Esp*{\varphi'(X)} - \Esp*{X \varphi(X)} : \norm{\varphi}_{\infty} \leq \sqrt{\frac{\pi}{2}},\, \norm{\varphi'}_{\infty} \leq 2 },
\end{equation*}
this concludes the proof in the case where $\Sigma = \mathrm{id}$. 
Now we take an general positive, hence invertible, covariance matrix $\Sigma$.
On the one hand, by the stability of Wasserstein distances under the Lipschitz map $\mathbb{R}^{n} \ni v \mapsto \Sigma v$, we find that
\begin{equation*}
  \mathbf{W}_{p}(X, \Sigma N) \leq \norm{\Sigma}_{op} \mathbf{W}_{p}(\Sigma^{-1}X,N).
\end{equation*}
Recall that $\norm{\cdot}_{op}$ is the operator norm with respect to the Euclidean norm on $\mathbb{R}^{n}$.
On the other hand, a routine computation yields
\begin{equation*}
  \Gamma[\Sigma^{-1} X, - \Sigma^{-1}F] = \Sigma^{-1} \Gamma[X, -F] \Sigma^{-1}.
\end{equation*}
Thus by \cref{eq:stein-bound-wasserstein}, we get
\begin{equation}\label{eq:stein-bound-wasserstein:covariance}
  \begin{split}
    \mathbf{W}_{p}(X, \Sigma N) & \leq \norm{\Sigma}_{op} \norm{N}_{L^{p}} \norm{ \mathrm{id} - \Sigma^{-1} \Gamma[X, -F] \Sigma^{-1} }_{L^{p}} + \norm{\Sigma}_{op} \norm{Z}_{L^{p}}
                              \\& \leq \norm{\Sigma}_{op} \norm{C^{-1}} \norm{N}_{L^{p}} \norm{C - \Gamma[X,-F]}_{L^{p}} + \norm{\Sigma}_{op} \norm{Z}_{L^{p}}.
  \end{split}
\end{equation}
Similarly, in the univariate case $d=1$, we obtain the bound \cref{eq:stein-bound-tv}.
\end{proof}

\section{The master equation for linear statistics}\label{s:master-equation}
We now apply the content of the previous section to the specific operator $\mathsf{L}$ which is the generator of the Dyson Brownian motion which was given in equation \eqref{eq:generator}.
\subsection{Notations and introduction of relevant operators}\label{master-operator-def}
We show here that linear statistics associated with $\beta$-ensembles can indeed be expressed in the form $\mathsf{L}F + Z$ thus making them good candidates to apply our \cref{th:stein-approximation}.
For a more concise exposition of our results and proofs, we adopt the following notations related to the empirical measure:
\begin{align*}
  & \nu_n := \sum_{i=1}^n \delta_{\lambda_i},\ \text{and}\quad \ \bar{\nu}_{n} := \nu_n-n\mu_V;
\\& \mu_n := \frac{\nu_n}{n}, \ \text{and}\quad \ \bar{\mu}_{n} := \mu_n-\mu_V.
\end{align*}
We will write $\pairing{f}{\mu}$ for the integration of a function $f$ against a measure $\mu$.
We consider the following operators, acting on functions $f \in \mathscr{C}^{2}(\mathbb{R})$:
\begin{align*}
    & T_V(f)(x) := \int \frac{f(x)-f(y)}{x-y} \mathtt{d} \mu_V(y) = \int \int_{0}^{1} f'((1-u)x + u y) \mathtt{d} u \mu_{V}(\mathtt{d} y);
  \\& \Theta_V(f) := -V' f + T_V(f);
  \\& T_n(f)(x) := \int \int_{0}^{1} f''((1-u)x + u y) \mathtt{d} u \mu_{n}(\mathtt{d} y) = \frac{1}{n} \sum_{i=1}^{n} \frac{f'(x) - f'(\lambda_{i})}{x - \lambda_{i}} 1_{x \ne \lambda_{i}} + f''(\lambda_{i}) 1_{x = \lambda_{i}};
  \\& \Theta_n(f) := -V' f' + T_n(f).
\end{align*}

To some extend, $T_n$ and $\Theta_n$ are the empirical counterparts of $T_V$ and $\Theta_V$.
Since $\mu_V$ is compactly supported and since $f$ is chosen $\mathscr{C}^2$, the integrals above are all convergent. This shows in particular that the operators $T_V$ and $\Theta_V$ are well-defined. We recall an \emph{invertibility result} for $\Theta_{V}$ taken from \cite{BLS}.

\begin{lemma}[{\cite[Lem.\ 3.3]{BLS}}]\label{th:inverse-theta}
There exists $\delta > 0$, such that if $U := (-1-\delta,1+\delta)$, for all $\xi \in \mathscr{C}^{6}(\mathbb{R})$, there exist a constant $c_{\xi}$ and $\psi \in \mathscr{C}^{5}_{c}(\mathbb{R})$ such that
\begin{align}
  & \label{eq:inverse-theta} (\Theta_{V}\psi)(x) = \xi(x) + c_{\xi}, \qquad x \in U,
\\& \label{eq:bound-inverse-theta} \norm{\psi}_{\mathscr{C}^{5}(\mathbb{R})} \leq C \norm{\xi}_{\mathscr{C}^{6}(\mathbb{R})}.
\end{align}
\end{lemma}

For the rest of the paper, we fix the margin $\delta>0$ and the associated neighborhood $U$ given by Lemma \ref{th:inverse-theta} above.

\begin{remark}
  \begin{enumerate}[(a),wide]
\item A slight variant of this invertibility result can be found in \cite[Lem.\ 3.2]{BGF}.
    \item We do not assume \cite[(H3)]{BLS}, however as noticed in this last reference, this assumption is not used in the proof of \cite[Lem.\ 3.3]{BLS}.
\item Since we are in the single-cut case, with no singular density, with the notations of \cite{BLS}, we have here $\mathsf{k} = 0$ and $\mathsf{m} = 0$.
\end{enumerate}
\end{remark}

\subsection{The rigidity estimates and its consequences}

Many ingredients in our proof proceed from a rigidity result from \cite{BEY} which is also used in \cite{LLW}.

\begin{theorem}[{\cite[Thm 2.4]{BEY}}]\label{th:rigidity}
  Let $\rho_0 := -1$, and $\rho_{j}$ be the $j/n$-quantile of $\mu_{V}$ for $j\in\{1,\cdots,n\}$.
  Set $\hat{\jmath} := \min(j,n-j+1)$ and write $(\lambda_{(j)})$ for the increasingly-ordered vector of $\lambda_{j}$'s.
  For any $\varepsilon > 0$, there exist $c_{\varepsilon} >0$ and $N_{\varepsilon}>0$ such that for all $n \geq N_{\varepsilon}$
\begin{equation}\label{rigidity-equation}
  \Prob*{ \exists j\in \lbrack 1,n \rbrack : \abs{\lambda_{(j)} -\rho_{j}}  > \hat{\jmath}^{-\frac{1}{3}} n^{-\frac{2}{3} + \varepsilon} }\le \exp(-n^{c_{\varepsilon}}).
\end{equation}
\end{theorem}

As anticipated, we repeatedly use in our proofs, the following corollary on the control of the outliers.
Philosophically, it allows us to work on $U = (-1-\delta, 1+\delta)$ and to discard the rest up to accepting a negligible loss.
The next result follows directly from the Theorem \ref{th:rigidity} and fact that the locations $(\rho_{j})_{1\le j \le n}$ belong to $[-1,1]$.

\begin{lemma}\label{th:outlier-control}
  There exists $C > 0$ and $c > 0$ such that for $n$ large enough
\begin{equation}\label{eq:outlier-control}
\Prob*{ \max_{1\le i \le n}|\lambda_i|\ge 1+\delta }\le C \Exp*{-n^{c}}.
\end{equation}
\end{lemma}

\begin{remark}
  Similar results are numerous in the literature: even with weaker assumptions on the potential $V$ than the ones in \cite{BEY} for which the rigidity is not yet known, one can derive better rate of convergence, such as exponential decay. Such exponential decay follows, for instance, from the large deviations principle for the extreme positions, as in \cite[Prop.\ 2.1]{BorotGuionnet}.
Under strong assumptions on the potential $V$, the reference \cite[Lem.\ 4]{MMS} offers an exponential decay with an explicit dependence on $\delta$, which, of course we do not need since our margin $\delta$ is fixed.
See also \cite[Thm.\ 1.12]{CHM} for similar a similar result under weaker assumptions in dimension greater than two. 
As the rigidity estimate given by Theorem \ref{th:rigidity} is in fact used to quantify the speed of convergence to equilibrium in the next Lemma \ref{th:llg-quantitative} and since the associated bound \cref{eq:outlier-control} is sufficient for our purpose, we have decided to use the rigidity also to control the outliers.
\end{remark}

As proved in \cite{LLW}, \cref{th:rigidity} provides a polynomial speed of convergence to equilibrium.
Namely, they obtain the following lemma that we reproduce below for completeness.

\begin{lemma}[{\cite[Lemma 5.3]{LLW}}]\label{th:llg-quantitative}
Let $\alpha>0$ and $p\ge 1$.
Then, there exists $K_{p,\alpha}>0$ such that for any bounded and Lipschitz function $f$ we have
\begin{equation}\label{eq:cvgce-equilibrium}
  \norm{ \pairing{f}{\mu_n - \mu_{V}}}_{L^{p}} \le \frac{1}{n^{1-\alpha}} K_{p,\alpha} (\norm{f}_\infty + \norm{f'}_\infty) .
\end{equation}
\end{lemma}

\begin{remark}\label{rmk:bootstrap}As we shall see below, our findings allow to upgrade the rate of convergence given in Lemma  \ref{th:llg-quantitative} to get an optimal rate $O(n^{-1})$ for regular enough test functions $f$.
\end{remark}

\subsection{The master equation}
In relation to the Dyson generator $\mathsf{L}$, the master operator $\Theta_{V}$ allows us to derive the following \emph{master equation}, which is at the heart of our argument.

\begin{theorem}\label{th:master}
Consider a test function $f \in \mathscr{C}^{2}(\mathbb{R})$ and define $F := \pairing{f}{\bar{\nu}_{n}}$. Then, we have the following decomposition
  \begin{equation}\label{eq:master-L}
    \frac{\mathsf{L}F}{n} = \frac{2-\beta}{2} \pairing{f''}{ \mu_n} + \beta  \pairing{\Theta_Vf'}{\bar{\nu}_{n}} + \frac{\beta}{2}  \pairing{T_nf'-T_Vf'}{ \bar{\nu}_{n}}.
  \end{equation}
\end{theorem}

\begin{proof}
  Recalling the definitions of $\mathsf{L}$ and $T_{n}$, we get that

  \begin{equation}\label{eq:LF-1}
\begin{split}
    \frac{\mathsf{L}F}{n} &= \pairing{f''}{\mu_n} - \beta \pairing{V' f'}{\nu_n} + \frac{\beta}{2 n} \sum_{i\neq j}^n \frac{f'(\lambda_i)-f'(\lambda_j)}{\lambda_i-\lambda_j}
                        \\&= \frac{2-\beta}{2} \pairing{f''}{\mu_n} - \beta \pairing{V' f'}{\nu_n} + \frac{\beta}{2} \pairing{T_{n}f'}{\nu_{n}}
                        \\&= \frac{2-\beta}{2} \pairing{f''}{\mu_n} -\beta \pairing{V' f'}{\bar{\nu}_{n}} + \frac{\beta}{2} \pairing{T_nf'}{\nu_n} -n\beta \pairing{V' f'}{\mu_V}.
  \end{split}
\end{equation}

Recall that the Euler--Lagrange equation for the minimality of $\mu_V$ with respect to the energy reads, for some $C_{V} \in \mathbb{R}$:
\begin{equation*}
  V(x)- \int \log(|x-y|) \mu_V(\mathtt{d} y) = C_V, \qquad x \in \Sigma_{V}.
\end{equation*}
Differentiating with respect to $x$ leads to 
\begin{equation*}
  V'(x)-\int \frac{1}{x-y} \mu_V(\mathtt{d}y) = 0, \qquad x \in \Sigma_{V}.
\end{equation*}
From there, multiplying by $f'(x)$, integrating with respect to $\mu_{V}(\mathtt{d}x)$, and symmetrizing yields
\begin{equation*}
  \pairing{V'f'}{\mu_V} = \frac{1}{2}\int \frac{f'(x)-f'(y)}{x-y} \mu_V(\mathtt{d}x)\mu_V(\mathtt{d}y) = \frac{1}{2} \pairing{T_Vf'}{\mu_V}.
\end{equation*}

Substituting the latter in \cref{eq:LF-1}, one gets
\begin{equation}\label{eq:LF-2}
  \begin{split}
    \frac{\mathsf{L}F}{n} &= \frac{2-\beta}{2} \pairing{f''}{ \mu_n} - \beta \pairing{V' f'}{\bar{\nu}_{n}}
    + \frac{\beta}{2}  \pairing{T_nf'}{\nu_n} - \frac{\beta}{2}  \pairing{T_Vf'}{n \mu_V}
    \\&= \frac{2-\beta}{2} \pairing{f''}{ \mu_n} - \beta \pairing{V' f'}{\bar{\nu}_{n}} - \frac{\beta}{2} \pairing{T_Vf'}{n \mu_V}
    \\&+ \frac{\beta}{2} \pairing{T_nf'-T_Vf'}{ \bar{\nu}_{n}} + \frac{\beta}{2} \pairing{T_nf'}{n \mu_V} + \frac{\beta}{2}  \pairing{T_Vf'}{\bar{\nu}_{n}}.
\end{split}
\end{equation}

By Fubini Theorem, we have that $\pairing{T_n f'}{n\mu_V} =  \pairing{T_Vf '}{\nu_n}$.
This gives
\begin{equation*}
  \pairing{T_nf'}{n \mu_V}- \pairing{T_Vf'}{n \mu_V}=  \pairing{T_Vf'}{\bar{\nu}_{n}}.
\end{equation*}
Plugging this equality in \cref{eq:LF-2} leads to the announced result.
\end{proof}



\section{Proofs of the main results}\label{sec.proofs}
We now complete the proofs of our main results. Regarding quantitative normal approximation for linear statistics of $\beta$-ensembles,  we follow the strategy described in Sections \ref{sec.stein} and \ref{s:master-equation}, applying the general quantitative bounds given by \cref{th:stein-approximation}.

\subsection{Proof of Theorem \ref{th:normal-approximation}} We first give the proof of Theorem \ref{th:normal-approximation}, which gives a near optimal rate of convergence $O(n^{-1+\alpha})$ for all $\alpha>0$.

\subsubsection{Handling the covariance} \label{sec.invert}
Let us recall the form of the limit covariance matrix $C$ of the linear statistics and let us first prove that it is invertible.
Following \cite[\S 4.1, in particular Eq.\ (4.14)]{LLW}, we have 
\begin{equation*}
c_{ij} = \frac{1}{2\beta} \int \frac{\xi_{i}(x) - \xi_{i}(y)}{x - y} \frac{\xi_{j}(x) - \xi_{j}(y)}{x - y} (1 - xy) \rho(\mathtt{d}x)\rho(\mathtt{d}y) =: \psh{\xi_{i}}{\xi_{j}}_{\mathscr{H}^{1/2}}.
\end{equation*}
The bilinear symmetric form $\psh{\xi_{i}}{\xi_{j}}_{\mathscr{H}^{1/2}}$ is not a scalar product, since $\psh{g}{g}_{\mathscr{H}^{1/2}} = 0$ if and only if $g$ is constant.
Nevertheless,  the matrix $C := (c_{ij})$ is a Gram matrix with respect to the semi-scalar product $\psh{\cdot}{\cdot}_{\mathscr{H}^{1/2}}$.
Thus under the freeness condition \cref{eq:free}, the matrix $C$ is symmetric definite positive, and we write $\Sigma$ for its unique positive symmetric square root.

\subsubsection{Preparatory computations}\label{par:preparatory} We now use the master equation \eqref{eq:master-L} to decompose the linear statistics in a suitable form to apply \cref{th:stein-approximation}.
\par
\medskip
\underline{Reduction to the case \texorpdfstring{$\mathsf{L}F + Z$}{LF + Z}}.
Let us consider $\xi \in \mathscr{C}^{6}(\mathbb{R})$, $\psi \in \mathscr{C}^{5}_{c}(\mathbb{R})$ associated with $\xi$ through \cref{th:inverse-theta} and let $f$ be any primitive of $\psi$.
We then have
\begin{equation*}
  \Theta_{V} f' = \xi + c_{\xi} + \paren*{\Theta_{V}f' - \xi - c_{\xi}} 1_{\mathbb{R} \setminus U}.
\end{equation*}
Since $\pairing{1}{\bar{\nu}_{n}} = 0$, and that $\operatorname{supp} \mu_{V} \subset U$, we find that
\begin{equation*}
  \pairing{\xi}{\bar{\nu}_{n}} = \pairing{\Theta_{V}f'}{\bar{\nu}_{n}} + \sum_{i=1}^{n} 1_{\abs{\lambda_{i}} > 1 + \delta} (\Theta_{V}f' - \xi - c_{\xi}).
\end{equation*}
By \cref{th:outlier-control}, we have that
\begin{equation*}
  \norm*{\sum_{i=1}^{n} 1_{\abs{\lambda_{i}} > 1 + \delta} (\Theta_{V}f' - \xi - c_{\xi})}_{L^{p}}=n \mathrm{e}^{-n^{c}} O\left( \norm{\xi}_{\infty} + \norm{\Theta_{V}f'}_{\infty} + 1\right).
\end{equation*}
From the explicit expression of $\Theta_{V}$ and \cref{eq:bound-inverse-theta}, we find that
\begin{equation*}
  \norm{\Theta_{V}f'}_{\infty} \leq c (\norm{f'}_{\infty} + \norm{f''}_{\infty}) \leq c \norm{\psi}_{\mathscr{C}^{5}(\mathbb{R})} \leq c \norm{\xi}_{\mathscr{C}^{6}(\mathbb{R})}.
\end{equation*}
Combining those two estimates, and since $\pairing{1}{\bar{\nu}_{n}} = 0$, it is sufficient to establish our bound for linear statistics of the form $\pairing{\Theta_{V}f'}{\bar{\nu}_{n}}$, where $f' = \psi \in \mathscr{C}_{c}^{5}(\mathbb{R})$ is associated with $\xi \in \mathscr{C}^{6}(\mathbb{R})$ by \cref{th:inverse-theta}.
Thus, we now study
\begin{equation*}
  X := \paren*{ \pairing{\Theta_{V}f_{1}'}{\bar{\nu}_{n}}, \dots, \pairing{\Theta_{V}f_{d}'}{\bar{\nu}_{n}} }.
\end{equation*}
By \cref{th:master}, we can then decompose $X$ under the form
\[
X = m+\frac{1}{n} \mathsf{L} F+ Z,
\]
where $m=(m_i)_{1\leq i\leq d}$ with $m_i:=(1/2-1/\beta)\pairing{f_{i}''}{\mu_V}$ where 
$F_{i} := \frac{1}{\beta} \pairing{f_{i}}{\bar{\nu}_{n}}$, and
\begin{equation}\label{eq:Z}
  Z_{i} := \paren*{ \frac{1}{2} - \frac{1}{\beta} } \pairing{f_{i}''}{\mu_{n}}-m_i - \frac{1}{2} \pairing{T_{n} f_{i}' - T_{V} f_{i}'}{\bar{\nu}_{n}}.
\end{equation}
In the next section, we apply \cref{th:stein-approximation} to provide quantitative bounds.
\par
\bigskip
\underline{Computation of \texorpdfstring{$\Gamma[X,-F]$}{the carré du champ}}.
We repeatedly use the simple yet remarkable observation that linear statistics are stable under the action of the carré du champ, namely
\begin{equation}\label{eq:gamma-linear-statistics}
\Gamma\bracket[\big]{\pairing{\varphi}{\nu_n}, \pairing{\psi}{\nu_n}} = \pairing{\varphi' \psi'}{\nu_n}, \qquad \varphi, \psi \in \mathscr{C}^{1}(\mathbb{R}).
\end{equation}
Thus,  we find that for $1\leq i , j\leq d$
\begin{equation}\label{eq:gamma-X-F}
  \Gamma\bracket*{X_{i},-\frac{F_{j}}{n}} = \frac{1}{\beta} \pairing{(\Theta_{V}f_{i}')'f_{j}'}{\mu_{n}}.
\end{equation}

\par
\bigskip
\underline{Heuristics.}
As a consequence of the convergence to equilibrium, $\mu_{n} \to \mu_{V}$, we see that the sum of the first two terms on the right-hand side in \cref{eq:Z} converges to zero.
Similarly, by Equation \cref{eq:gamma-X-F}, the term $\Gamma[X_{i}, - F_{j}/n]$ converges to $\frac{1}{\beta}\pairing{\xi_{i}' f_{j}'}{\mu_{V}}$ which will be the limit covariance of $(\langle \xi_i,\bar{\nu_n}\rangle,\langle \xi_j,\bar{\nu_n}\rangle)$. Relying on the expression of the limit variance given in \cite{LLW} and by uniqueness one must therefore have
$$c_{ij}=\lim_n\text{Cov}(\langle \xi_i,\bar{\nu_n}\rangle,\langle \xi_j,\bar{\nu_n}\rangle)=\frac{1}{\beta}\pairing{\xi_{i}' f_{j}'}{\mu_{V}}.$$

However, the determination of the asymptotics of all the terms involving the quadratic term $\pairing{(T_{n}-T_{v}) f_{i}'}{\bar{\nu}_{n}}$ is more delicate.
Indeed, the same heuristic on the convergence to equilibrium shows that this term is of order $o(n)$ whereas one would expect $o(1)$.
We handle this remainder in two steps:
\begin{enumerate}[(i), wide]
  \item We split the quadratic term in a product of linear terms, using Fourier inversion.
    Indeed since $f_{i}' \in \mathscr{C}_{c}^{5}(\mathbb{R})$ is is in the image of the Fourier transform.
  \item We control each of the linear term as before using convergence to equilibrium.
    This is where we need the finer quantitative estimates recalled in \cref{th:llg-quantitative}.
\end{enumerate}

\subsubsection{Quantitative control and completion of the proof} Le us now quantify the remainders mentioned in the last Section. 
\par
\medskip
\underline{Splitting the remainder by Fourier inversion.}
For simplicity, we omit the indices here and write $f$ for $f_i$, $1\leq i \leq d$. 
Since $f' \in \mathscr{C}^{5}_{c}(\mathbb{R})$, we have
\begin{equation}\label{eq:bound-fourier-c4}
  \abs{ \widehat{f''}(t) } \le 2 \frac{\norm{f^{(6)}}_{\infty}}{(1+|t|)^{4}}, \qquad t \in \mathbb{R}.
\end{equation}
In particular, $\widehat{f''}  \in L^1(\mathbb{R})$.
By the Fourier inversion Theorem, one finds
\begin{equation}\label{eq:fourier-splitting}
  \begin{split}
    \pairing{(T_{n} - T_{V})f'}{\bar{\nu}_{n}} &= n \iint \int_{0}^{1} f''((1-u) x + u y) (\mu_{n} - \mu_{V})(\mathtt{d} x) (\mu_{n} - \mu_{V})(\mathtt{d} y)
                                                     \\&= n \int_{\mathbb{R}} \int_{0}^{1} \widehat{f''}(t) \pairing{\mathrm{e}^{\mathrm{i}tu \bullet}}{\mu_{n} - \mu_{V}} \pairing{\mathrm{e}^{\mathrm{i}t(1-u) \bullet}}{\mu_{n} - \mu_{V}} \mathtt{d} u \mathtt{d} t.
  \end{split}
\end{equation}
By \cref{th:llg-quantitative} applied to the exponential functions and by Hölder's inequality, we have
\begin{equation*}
  n \norm{\pairing{\mathrm{e}^{\mathrm{i}tu \bullet}}{\mu_{n} - \mu_{V}} \pairing{\mathrm{e}^{\mathrm{i}t(1-u) \bullet}}{\mu_{n} - \mu_{V}}}_{L^{p}} \leq \frac{1}{n^{1-2\alpha}} K_{p,\alpha}^{2} (1 + \abs{t} +t^{2}).
\end{equation*}
Thus, reporting in \cref{eq:fourier-splitting}, and up to changing the constants from line to line, we find that
\begin{equation}\label{eq:remainder-bulk-Lp}
  \norm{\pairing{(T_{n} - T_{V})f'}{\bar{\nu}_{n}}}_{L^{p}} \leq \frac{1}{n^{1-2\alpha}} K_{p,\alpha} \int_{\mathbb{R}} \abs{\widehat{f''}(t)} \paren*{1 + \abs{t} + t^2} \mathtt{d} t \leq \frac{1}{n^{1-2\alpha}} K_{p,\alpha} \norm{f^{(6)}}_{\infty}.
\end{equation}

\par
\medskip
\underline{Completion of the proof.}
We can finally complete the proof {\cref{th:normal-approximation}}. Again, to simplify the expressions, we omit the indices $1\leq i\leq d. $ here. Recall the definition of the term $Z$ given by Equation \cref{eq:Z}.
By \cref{th:llg-quantitative} and the above Equation \eqref{eq:remainder-bulk-Lp}, we find that
\begin{equation*}
  \norm{Z}_{L^{p}} \leq \paren*{\frac{1}{2} - \frac{1}{\beta}} \norm{\pairing{f_{}''}{\mu_{n} - \mu_{V}}}_{L^{p}} + \norm{\pairing{(T_{n} - T_{V})f'}{\bar{\nu}_{n}}}_{L^{p}} \leq \frac{1}{n^{1-\alpha}} K_{p,\alpha} \norm{f'}_{\mathscr{C}^{5}(\mathbb{R})}.
\end{equation*}
Introducing the random matrix
\begin{equation*}
 C_{n}(i,j) := \pairing{\xi'_{i} f'_{i}}{\mu_{n}},
\end{equation*}
 we have by Equation \cref{eq:gamma-X-F}
\begin{equation*}
  \left \|\frac{1}{\beta} C - \Gamma[X, -F/n]\right \|_{L^{p}} = \frac{1}{\beta} \norm{C - C_{n}}.
\end{equation*}
Invoking \cref{th:llg-quantitative} to control $C - C_{n}$
yields, again with constants which may change from line to line,
\begin{equation}\label{eq:bound-approximation-covariance}
  \begin{split}
    \norm{C - C_{n}}_{L^{p}} & \leq K_{p,\alpha} \frac{1}{n^{1-\alpha}}  \paren*{ \norm{\xi'}_{\infty} \norm{f'}_{\infty} + \norm{\xi'}_{\infty} \norm{f''}_{\infty} + \norm{\xi''}_{\infty} \norm{f'}_{\infty} } 
                           \\& \leq \frac{K_{p,\alpha}}{n^{1-\alpha}} \norm{\xi}_{\mathscr{C}^{2}(\mathbb{R})} \norm{f}_{\mathscr{C}^{2}(\mathbb{R})}.
  \end{split}
\end{equation}

Since by \cref{th:inverse-theta}, we have $\norm{f}_{\mathscr{C}^{2}(\mathbb{R})} \leq K\norm{\xi}_{\mathscr{C}^{2}(\mathbb{R})}$, all the bounds above can thus be rewritten only in term of the norm $\norm{\xi}_{\mathscr{C}^{2}(\mathbb{R})}$. Injecting these estimates in the general upper bound given by Theorem \ref{th:stein-approximation}, we arrive at the announced result.

\subsection{Proof of Theorem \ref{th:normal-approximation-optimal}}
We now give the proof of Theorem \ref{th:normal-approximation-optimal}, which gives an optimal rate of convergence $O(n^{-1})$ in the case where the test function $\xi$ is sufficiently regular. We work here in dimension $d=1$ for simplicity, but the proof will extend verbatim to higher dimensions. \par
\bigskip
The strategy used in order to obtain the optimal rate of convergence consists in bootstrapping the argument used in the previous proof of \cref{th:normal-approximation}. 
Except that instead of invoking \cref{th:llg-quantitative}, we use here a better law of large numbers provided by \cref{th:normal-approximation} itself.
Namely, for $\xi \in \mathscr{C}^{6}(\mathbb{R})$, taking $\alpha := 1/2$ and using the $p$-Wasserstein bound for any $p\ge 1$ given by \cref{th:normal-approximation}, we have that
\begin{equation*}
\norm{\pairing{\xi}{\bar{\nu}_{n}}}_{L^{p}} \leq K_p \left(\frac{\norm{\xi}_{\mathscr{C}^{6}(\mathbb{R})})+ \norm{\xi}_{\mathscr{C}^{2}(\mathbb{R})}^{2}}{n^{1/2}} \right)+ \norm{\sigma N +m}_{L^{p}},
\end{equation*}
where $N$ is a standard Gaussian and both the mean $m=m_1$ and the standard deviation $\sigma=\sqrt{c_{11}}$ are controlled by $\|\xi\|_{\mathscr{C}^1(\mathbb{R})}$, see Section \ref{sec:notations} where their explicit expressions are given.
This shows that provided we can control the $6$-th derivative, we can remove the polynomial loss in the quantitative law of large numbers given by Lemma \ref{th:llg-quantitative}.
At the level of the upper bound \cref{eq:remainder-bulk-Lp}, this implies that we have a $t^{12}$ appearing in the integral, which leads to an upper bound controlled by $\norm{\xi}_{\mathscr{C}^{14}(\mathbb{R})}$.
Similarly, in the upper bound \cref{eq:bound-approximation-covariance}, the right hand side now yields a control of the form $\norm{\xi'}_{\mathscr{C}^{6}(\mathbb{R})} \norm{f'}_{\mathscr{C}^{6}(\mathbb{R})}$ and concludes the proof.

\medskip

As a result and as already mentioned in Remark \ref{rmk:bootstrap}, this argument implies that any polynomial speed of convergence on the law of large numbers can be upgraded to the optimal speed $\frac{1}{n}$, the price to pay is to impose stronger regularity assumption on the test functions.

\subsection{Proof of Theorem \ref{th:normal-approximation-lipschitz} }

We now show how \cref{th:normal-approximation-lipschitz} dealing with functions with H\"{o}lder derivative easily follows from \cref{th:normal-approximation} together with \cref{th:llg-quantitative}.
\par
\bigskip 
Take $\eta$ a smooth probability density supported on $[-1,1]$ with finite first moment.
Defining $\eta_{\varepsilon} := \frac{1}{\varepsilon} \eta(\frac{\bullet}{\varepsilon})$, and $\xi_{\varepsilon} := \xi \ast \eta_{\varepsilon}$, we see by direct computations that
\begin{align}\label{eq:convol-reste}
  \nonumber& \norm{\xi - \xi_{\varepsilon}}_{\infty} \leq \norm{\xi'}_{\infty} \varepsilon \int \abs{y} \eta(y),
\\& \norm{\xi' - \xi_{\varepsilon}'}_{\infty} \leq \norm{\xi}_{\mathscr{C}^{1,\gamma}(\mathbb{R})} \,\varepsilon^{\gamma} \int \abs{y}^{\gamma} \eta(y).
\end{align}
Thus by \cref{th:llg-quantitative}, for any $\alpha>0$, we have
\begin{equation*}
  \norm{\pairing{\xi - \xi_{\varepsilon}}{\bar{\nu}_{n}}}_1 \leq K_{\alpha} \varepsilon^{\gamma} n^{\alpha}\norm{\xi}_{\mathscr{C}^{1,\gamma}(\mathbb{R})}.
\end{equation*}

Now we use the master equation \eqref{eq:master-L} for the linear statistics $\xi_\varepsilon$ and we write:
$$\frac{\mathsf{L}F_\varepsilon}{n} = \frac{2-\beta}{2} \pairing{f_\varepsilon''}{ \mu_n} + \beta  \pairing{\Theta_v f_\varepsilon'}{\bar{\nu}_{n}} + \frac{\beta}{2}  \pairing{T_nf_\varepsilon'-T_Vf_\varepsilon'}{ \bar{\nu}_{n}},$$

where $\Theta_v f_\varepsilon'=\xi_\varepsilon$ on the neighbourhood $U$. Mimicking the arguments developed in Section \ref{par:preparatory}, we thus have:
$$\frac{\mathsf{L}F_\varepsilon}{n} = \frac{2-\beta}{2} \pairing{f_\varepsilon''}{ \mu_n} + \beta  \pairing{\xi_\varepsilon}{\bar{\nu}_{n}} + \frac{\beta}{2}  \pairing{T_nf_\varepsilon'-T_Vf_\varepsilon'}{ \bar{\nu}_{n}}+O\left(ne^{-n^c} \|\xi_\varepsilon\|_{\mathscr{C}^6}\right).$$

As a result, we get
\[
\frac{\mathsf{L}F_\varepsilon}{n} = \frac{2-\beta}{2} \pairing{f_\varepsilon''}{ \mu_n} + \beta  \pairing{\xi}{\bar{\nu}_{n}} + \frac{\beta}{2}  \pairing{T_nf_\varepsilon'-T_Vf_\varepsilon'}{ \bar{\nu}_{n}}+O\left(ne^{-n^c} \|\xi_\varepsilon\|_{\mathscr{C}^6}\right)
+\pairing{\xi - \xi_{\varepsilon}}{\bar{\nu}_{n}}.
\]


Hence, we may decompose the linear statistics in the form $\pairing{\xi}{\bar{\nu}_{n}}=m_\varepsilon+\frac{\mathsf{L}F_\varepsilon}{n}+Z_\varepsilon$ as in the proof of Theorem \ref{th:normal-approximation} with this time 

\[
\begin{array}{ll}
\displaystyle{m_{\varepsilon}:= \left( \frac{1}{2}-\frac{1}{\beta}\right)\pairing{f_{\varepsilon}''}{\mu_V}}, \\
\\
\displaystyle{Z_\varepsilon:=\frac{2-\beta}{2} \pairing{f_\varepsilon''}{ \mu_n-\mu_V}+\frac{\beta}{2}  \pairing{T_nf_\varepsilon'-T_Vf_\varepsilon'}{ \bar{\nu}_{n}}+O\left(ne^{-n^c} \|\xi_\varepsilon\|_{\mathscr{C}^6}\right)+\pairing{\xi - \xi_{\varepsilon}}{\bar{\nu}_{n}}.}
\end{array}
\]

Gathering the bounds given in the proof of Theorem \ref{th:normal-approximation} and the above Equation \eqref{eq:convol-reste}, we then obtain the following controls
\begin{eqnarray*}
&&\left\|\frac{2-\beta}{2} \pairing{f_\varepsilon''}{ \mu_n-\mu_V}+\frac{\beta}{2}  \pairing{T_nf_\varepsilon'-T_Vf_\varepsilon'}{ \bar{\nu}_{n}}\right\|_1\le \frac{K_\alpha}{n^{1-2\alpha}}\left(\|\xi_\varepsilon\|_{\mathscr{C}^6}+\|\xi_\varepsilon\|_{\mathscr{C}^2}^2\right),\\
\\
&&\|\pairing{\xi - \xi_{\varepsilon}}{\bar{\nu}_{n}}\|_{1}\le K_\alpha\varepsilon^{\gamma} n^{\alpha}\norm{\xi}_{\mathscr{C}^{1,\gamma}(\mathbb{R})},\\
\\
&&\norm{C_\varepsilon - C_{n,\varepsilon}}_{L^{p}}\le \frac{K}{n^{1-\alpha}}\|\xi_\varepsilon\|_{\mathscr{C}^2}^2,
\end{eqnarray*}
where $C_\varepsilon$ denote the limit variance associated with the linear statistic $\xi_\varepsilon$ and $C_{n,\varepsilon} := \pairing{\xi'_{\varepsilon} f'_{\varepsilon}}{\mu_{n}}$ is the empirical analogue. By Equation \eqref{eq:stein-bound-tv} in Theorem \ref{th:stein-approximation}, combined with the three above bounds we may write for some constant $K_{\alpha}$

\begin{equation}\label{eq:bound-tv-convol}
 \mathbf{TV}\left(\pairing{\xi}{\bar{\nu}_{n}},m_\varepsilon+\sigma_\varepsilon N\right)\le K_\alpha\left(\frac{1}{\varepsilon^6 n^{1-2\alpha}}+\varepsilon^\gamma n^{\alpha}\right)\norm{\xi}_{\mathscr{C}^{1,\gamma}(\mathbb{R})}.
\end{equation}

In order to complete the proof, one needs to show that $m_\varepsilon$ and $\sigma_\varepsilon$ converge towards $m$ and $\sigma$ and to estimate $mathbf{TV}\left(m_\varepsilon+\sigma_\varepsilon N,m+\sigma N\right)$ as $\varepsilon$ goes to zero. First of all, relying on the explicit expressions of the limit means and variances, we have indeed as $\varepsilon$ goes to zero
\[
|m_\varepsilon-m|=O\left(\|\xi_\varepsilon-\xi\|_{\mathscr{C}^1}\right)=O\left(\varepsilon^\gamma\|\xi\|_{\mathscr{C}^{1,\gamma}}\right), \quad \text{as well as}\quad |\sigma_\varepsilon-\sigma|=O\left(\varepsilon^\gamma\|\xi\|_{\mathscr{C}^{1,\gamma}}\right).
\]
Finally, using for instance Theorem 1.3 in \cite{devroye2018total}, we derive that 
\[
\mathbf{TV}\left(m_\varepsilon+\sigma_\varepsilon N,m+\sigma N\right)\le K\varepsilon^\gamma\|\xi\|_{\mathscr{C}^{1,\gamma}}.
\]
Optimizing in the parameters in Equation \eqref{eq:bound-tv-convol}, since $\alpha$ can be chosen arbitrarily small, we get the announced bound of $O(n^{-a})$ for any $a<\frac{\gamma}{\gamma+6}$.

\subsection{Proof of Theorem \ref{th:super-convergence}}\label{sec.super1}
We finally give the proof of Theorem \ref{th:super-convergence} stating the super-convergence of linear statistics of $\beta-$ensembles. If 
$X$ denotes our linear statistics, the proof is based on the fact that $\Gamma[X,X]$ admit negative moments.  This fact, combined with classical integration by parts techniques indeed ensure the convergence of densities. 
\subsubsection{Control of the negative moments of \texorpdfstring{$\Gamma$}{the carré du champ}}

Let us first establish the following lemma.
\begin{lemma}\label{lem.neg}
  Take $\xi \in \mathscr{C}^{1}(\mathbb{R})$ such that there exist $\alpha > 0$ such that
  \begin{equation*}
    \mathsf{Leb} \paren*{ x \in \mathbb{R} : \abs{\xi'(x)} \le \varepsilon } \lesssim \varepsilon^{\alpha}, \qquad \varepsilon > 0.
  \end{equation*}
  Then, there exist $\gamma > 0$, such that for $n$ large enough
  \begin{equation}\label{eq:negative-moment-f}
    \Prob*{ \frac{1}{n} \sum_{i=1}^{n} \xi'(\lambda_{i})^{2} \leq \varepsilon } \lesssim \varepsilon^{\gamma n}, \qquad \varepsilon > 0.
  \end{equation}
\end{lemma}

\begin{proof}[Proof of Lemma \ref{lem.neg}]
To simplify the expressions in this proof, we set $g := (\xi')^{2}$ and $m := \pairing{g}{\mu_{V}}$. We distinguish here two regimes depending on the magnitude of $\varepsilon$.\par
\medskip
\underline{The large deviation regime.} Take $\varepsilon < \frac{1}{2} m$, then by the large deviations principle for $\beta$-ensembles \cite[Thm.\ 2.3]{Serfaty} (originally proved in \cite{BenArousGuionnet}), there exist $c=c_{g} > 0$ such that, for $n$ large enough
\[
 \Prob*{ \pairing{g}{\mu_{n}} \leq \varepsilon }  \leq \Prob*{ \abs{\pairing{g}{\bar{\mu}_{n}}} \geq \frac{m}{2} }\leq \Exp*{-c n^{2}}.
\]
  In this case, in the regime $\varepsilon \geq \mathrm{e}^{-kn}$ where $k$ is some constant to be fixed later, we have
  \begin{equation*}
    \mathrm{e}^{-c n^{2}} = (\mathrm{e}^{-kn})^{n \frac{c}{k}} \leq \varepsilon^{n \frac{c}{k}}.
  \end{equation*}
\par
\medskip
\underline{The free energy regime.}
  Let us work in the regime $\varepsilon \leq \mathrm{e}^{-kn}$, where we recall that we have the freedom to choose $k$.
  By the explicit form of the density for the $\beta$-ensemble, we find that
  \begin{equation*}
    \begin{split}
      \Prob*{ \pairing{g}{\mu_{n}} \leq \varepsilon } &\leq \Prob*{ g(\lambda_{i}) \leq \varepsilon n, \, \forall i \in \lbrack 1, n \rbrack }
                                                    \\&= \frac{1}{Z_{n,\beta}} \int_{\mathbb{R}^{n}} \prod_{i=1}^{n} \mathbf{1}_{\set*{ g(\lambda_{i}) \leq n \varepsilon } } \mathrm{e}^{-\beta H_{n}(\lambda_{1}, \dots, \lambda_{n})} \mathtt{d} \lambda_{1} \dots \mathtt{d} \lambda_{n}.
    \end{split}
  \end{equation*}
  Thus applying Cauchy--Schwarz inequality and using the $\alpha-$regular condition on $\xi$ hence $g$, we obtain
  \begin{equation*}
    \Prob*{ \pairing{g}{\mu_{n}} \leq \varepsilon } \leq \frac{Z_{n,2\beta}^{1/2}}{Z_{n,\beta}} (n\varepsilon)^{n\alpha/2}.
  \end{equation*}
By the large deviations principle, see for example \cite[Thm.\ 2.3]{Serfaty}, we know that
\begin{equation*}
  \frac{1}{n^{2}} \log Z_{n,\beta} \xrightarrow[n \to \infty]{} - \frac{\beta}{2} \mathcal{I}_{V}(\mu_V).
\end{equation*}
In particular, there exists a constant $a(\beta)$ such that
\begin{equation*}
  \frac{Z_{n,2\beta}^{1/2}}{Z_{n,\beta}} \leq \mathrm{e}^{n^{2} a(\beta)}.
\end{equation*}
As a result, we have
\begin{equation*}
  \Prob*{ \pairing{g}{\mu_{n}} \leq \varepsilon } \leq (n \varepsilon)^{n\alpha/2} \mathrm{e}^{a(\beta) n^{2}}.
\end{equation*}
We can assume that $a(\beta) \geq 0$ otherwise the claim is trivial. In this regime, we have
\[
  \mathrm{e}^{a(\beta) n^{2}} \leq \varepsilon^{-a(\beta)n/k}, \qquad n^{n\alpha/2} \leq \varepsilon^{-\alpha/(2k) \log n}.
\]
It follows that
\begin{equation*}
  \Prob*{ \pairing{g}{\mu_{n}} \leq \varepsilon } \leq \varepsilon^{n (\alpha / 2 - a(\beta)/k) - \alpha/(2k) \log n}.
\end{equation*}
Choosing $k$ large enough so that $\alpha/2 - a(\beta)/k r > 0$ then yields and upper bound for the probability of order $\varepsilon^{c n}$, for some $c > 0$.
\end{proof}

Recall that if $X=\sum_{i=1}^n \xi(\lambda_i)$ is a linear statistics, then we have $\Gamma[X,X]= \sum_{i=1}^n \xi'(\lambda_i)^2$. As a result, given an exponent $\alpha>0$, we can write
\[
  \Esp*{\frac{1}{\Gamma[X,X]^{\alpha}}} = \int_{\mathbb R^+}   \Prob*{ \frac{1}{\Gamma[X,X]^{\alpha}} >t} dt = \int_{\mathbb R^+}   \Prob*{ \Gamma[X,X] < \frac{1}{t^{1/\alpha}}} dt.
\]
As a result, the last Lemma \ref{lem.neg} indeed ensures that $\Gamma[X,X]$ admits negative moments.

\subsubsection{Regularity and super convergence}

The proof of {\cref{th:super-convergence}} now follows from a well-known integration by parts procedure see \cite[\S 2.1]{HMPQuadratic}, or \cite[\S 3.2.2]{HMPSuper}, for details.
We only recall the first part of the proof to highlight that, in the setting of $\beta$-ensemble, the correct quantities to control are the negative moments of $\Gamma[X,X]/n$, and not those of $\Gamma[X,X]$, as it is the case in \cite{HMPQuadratic,HMPSuper}.
Take $\varphi \in \mathscr{C}^{1}_{c}(\mathbb{R})$ and write 
\begin{equation}\label{eq:ipp-density}
  \Esp*{ \varphi'(X) } = \Esp*{ \varphi(X) \paren*{\frac{\Gamma[X, \Gamma[X,X]]}{\Gamma[X,X]^{2}} - \frac{\mathsf{L}X}{\Gamma[X,X]}} }.
\end{equation}
By Equation \cref{eq:gamma-linear-statistics}, we find that
\[
\Gamma[X, X] = \pairing{(\xi')^{2}}{\nu_{n}}, \qquad \Gamma[X, \Gamma[X,X]] = 2 \pairing{(\xi')^{2} \xi''}{\nu_{n}}.
\]
As a result, as $n$ goes to infinity, we have 
\begin{equation*}
\frac{\Gamma[X, \Gamma[X,X]]}{\Gamma[X,X]^{2}} \to 0.
\end{equation*}
Now in view of the master equation \cref{eq:master-L}, we have the decomposition 

\begin{equation}\label{eq:gamma-LX-X}
  \begin{split}
    \Gamma\bracket*{\frac{\mathsf{L}X}{n},\frac{X}{n}} &= \frac{2-\beta}{2} \frac{1}{n} \pairing{\xi^{(3)} \xi'}{\mu_n} + \frac{1}{\beta} \pairing{(\Theta_{V}\xi')' \xi'}{\mu_n}
                                                             \\&-\frac{1}{2\beta}\Gamma\bracket*{ \pairing{(T_{V} - T_{n}) \xi'}{\nu_n},\frac{X}{n}}.
  \end{split}
\end{equation}

Using again the Fourier splitting, by the chain rule, dominated convergence, and \cref{eq:gamma-linear-statistics}, we find that
\begin{equation*}
  \begin{split}
    R_{n} & := \Gamma \bracket*{ \pairing{(T_{n} - T_{V})\xi'}{\bar{\nu}_{n}}, \pairing{\psi}{\mu_{n}} }
                 \\&= \mathrm{i} \int_{\mathbb{R}} \widehat{\xi''}(t) t \int_{0}^{1} u \pairing{\mathrm{e}^{\mathrm{i} u \bullet} \psi'}{\mu_{n} - \mu_{V}} \pairing{\mathrm{e}^{\mathrm{i}t(1-u)\bullet}}{\mu_{n} - \mu_{V}} \mathtt{d} u \mathrm{d} t
                                                                                                    \\&+ \mathrm{i} \int_{\mathbb{R}} \widehat{\xi''}(t) t \int_{0}^{1} (1-u) \pairing{\mathrm{e}^{\mathrm{i} t(1-u) \bullet} \psi'}{\mu_{n} - \mu_{V}} \pairing{\mathrm{e}^{\mathrm{i}tu\bullet}}{\mu_{n} - \mu_{V}} \mathtt{d} u \mathrm{d} t.
  \end{split}
\end{equation*}
By \cref{th:llg-quantitative}, we then obtain that $R_{n} \to 0$ in $L^{2}$.
By \cref{th:stein-kernel-L}, this shows that $\mathsf{L}F/n$ converges to a Gaussian with respect to $\mathbf{W}_{2}$.
In particular, its variance is of order $1$.
Thus re-writing, the last term in \cref{eq:ipp-density} as the product
\begin{equation*}
  \frac{\mathsf{L}X}{n} \frac{n}{\Gamma[X,X]},
\end{equation*}
the need to control the negative moments of $\Gamma[X,X]/n$ is more obvious. It implies in particular that for some constant $C>0$ and any $\phi\in\mathscr{C}^1$ we have $\left|\mathbb{E}\left[\phi'(X)\right]\right|\le C \|\phi\|_\infty$. This argument can be iterated and for any $p>1$ one would get similarly for $n$ large enough that $\left|\mathbb{E}\left[\phi^{(p)}(X)\right]\right|\le C \|\phi\|_\infty$. The latter combined with the convergence in law proved in Theorem \ref{th:normal-approximation} entails the announced super-convergence, see \cite{HMPSuper} for more details.

\bibliographystyle{alpha}

\end{document}